\documentclass[11pt,english,reqno]{amsart}
\usepackage{bm}
\usepackage{color}
\usepackage{xcolor}
\usepackage{babel}
\usepackage{verbatim}
\usepackage{amsthm}
\usepackage{amstext}
\usepackage{amssymb}
\usepackage{amsbsy}
\usepackage{ifpdf}
\usepackage[nobysame,abbrev,alphabetic]{amsrefs}
\usepackage{enumerate}
\usepackage{amsmath}
\usepackage{amscd}
\usepackage{amsfonts}
\usepackage{graphicx}
\usepackage[all]{xy}
\usepackage{verbatim}
\usepackage{gensymb}
\usepackage{mathrsfs}
\usepackage[colorlinks]{hyperref}
\allowdisplaybreaks
\hypersetup{
linkcolor=blue,          
citecolor=green,        
}
\newtheorem{theorem}{Theorem}[section]

\newtheorem{lemma}[theorem]{Lemma}
\newtheorem{corollary}[theorem]{Corollary}
\theoremstyle{definition}\newtheorem{definition}[theorem]{Definition}

\theoremstyle{theorem}\newtheorem{proposition}[theorem]{Proposition}

\theoremstyle{definition}
\newtheorem{question}[theorem]{Question}
\theoremstyle{definition}
\theoremstyle{definition}\newtheorem{remark}[theorem]{Remark}
\theoremstyle{definition}\newtheorem*{acknowledgments}{Acknowledgments}
\newcommand{\al}{\alpha}

\newcommand{\Ga}{\Gamma}

\newcommand{\cM}{\mathcal{M}}

\newcommand{\cP}{\mathcal{P}}

\newcommand{\bP}{\mathbb{P}}
\newcommand{\bR}{\mathbb{R}}
\newcommand{\bZ}{\mathbb{Z}}

\newcommand{\SL}{\operatorname{SL}}
\newcommand{\SO}{\operatorname{SO}}

\newcommand{\GL}{\operatorname{GL}}

\newcommand\pa[1]{\left(#1\right)}

\newcommand\on[1]{\operatorname{#1}}

\newcommand\mb[1]{\mathbf{#1}}

\newcommand\br[1]{\left[#1\right]}
\newcommand\smallmat[1]{\pa{\begin{smallmatrix}#1\end{smallmatrix}}}

\newcommand{\Vol}{\operatorname{Vol}}
\newcommand{\diag}{\operatorname{diag}}
\newcommand{\sign}{\operatorname{sign}}
\newcommand{\Lie}{\operatorname{Lie}}
\newcommand{\Ad}{\operatorname{Ad}}
\newcommand{\Area}{\operatorname{Area}}

\newcommand{\onto}{\xymatrix{\ar@{>>}[r]&}}

\begin{document}
\title[Translates of Divergent Diagonal Orbits]{Limiting Distributions of Translates of Divergent Diagonal Orbits}
\author{Uri Shapira and Cheng Zheng}
\begin{abstract}
We define a natural topology on the collection of (equivalence classes up to scaling of) locally finite measures on a homogeneous 
space and prove that in this topology, pushforwards of certain infinite volume orbits equidistribute in the ambient space. As 
an application of our results we prove an asymptotic formula for the number of integral points in a ball on some varieties as the radius goes to infinity.
\end{abstract}
\address{Department of Mathematics\\
Technion \\
Haifa \\
Israel }
\email{ushapira@technion.ac.il}
\email{cheng.zheng@campus.technion.ac.il}
\thanks{}
\maketitle


\section{Introduction}\label{intro}
This paper deals with the study of the possible limits of periodic orbits in homogeneous spaces. Before explaining 
what we mean by this we start by motivating this study.
In many instances arithmetic properties of an object are captured by periodicity of a corresponding orbit in some dynamical system.
A simple instance of this phenomenon is that $\al\in \bR$ is rational if and only if its decimal expansion is eventually periodic. In dynamical terms this is expressed by the fact that the orbit of $\al$ modulo 1 on the torus $\bR/\bZ$
under multiplication by 10 (modulo 1) is eventually periodic. Furthermore, from knowing distributional information regarding the 
periodic orbit one can draw meaningful arithmetical conclusions. In the above example this means that if the orbit is very close to 
being evenly distributed on the circle then the frequency of appearance of say the digit 3 in the period of the decimal expansion is 
roughly $\frac{1}{10}$. This naive scheme has far reaching analogous manifestations capturing deep arithmetic concepts in dynamical terms. More elaborate instances are for example the following:
\begin{itemize}
\item Similarly to the above example regarding decimal expansion, periodic geodesics on the modular surface
correspond to continued fraction expansions of quadratic numbers and distributional properties of the former implies
statistical information regarding the latter (see~\cite{AS} where this was used).
\item Representing an integral quadratic form by another is related to periodic orbits of orthogonal groups (see~\cite{EV08}).
\item Class groups of number fields correspond to adelic torus orbits (see~\cite{ELMV}).
\item Values of rational quadratic forms are governed by the volume of
 periodic orbits of orthogonal groups (see~\cite[Theorem 1.1]{EMV09})
\item Asymptotic formulas for counting integer and rational points on varieties are encoded by distributional properties 
of periodic orbits (see~\cite{DRS93, EM93, EMS96, GMO08} for example). 
\end{itemize}
In all the above examples the orbits that are considered are of~\textit{finite volume}. Recently in~\cite{KK} and~\cite{OS14} this barrier was 
crossed and particular instances of the above principle were used for infinite volume orbits in order to obtain asymptotic estimates for counting integral points on some varieties and weighted second moments of $\GL(2)$ automorphic $L$-functions. 

At this point let us make more precise our terminology. Let $X$ be a locally compact second countable Hausdorff space and let $H$ be a unimodular topological group acting on $X$ continuously. We say that an orbit $Hx$ is \textit{periodic} if it supports an $H$-invariant \emph{locally finite} Borel measure. In such a case the orbit is necessarily closed and this measure is unique up to scaling and is obtained by restricting the Haar measure of $H$ to a fundamental domain of $\on{Stab}_H(x)$ in $H$ which is identified with the orbit via $h\mapsto hx$. We say that such an orbit is \textit{of finite volume} if the total mass of the orbit is finite. It is then customary to normalize
the total mass of the orbit to 1.  We remark that in some texts the term periodic orbit is reserved for finite volume ones but we wish to extend the terminology as above. If $Hx$ is a periodic orbit we denote by $\mu_{Hx}$ a choice of such a measure, which in the finite volume case is assumed to be normalized to a probability measure. 

Given a sequence of periodic orbits $Hx_i$, it makes sense to ask if  they converge in some sense to a limiting object. When the orbits are of finite volume the common definition is that of weak* convergence; each orbit is represented by the probability measure $\mu_{Hx_i}$ and one equips the space of probability measures $\cP(X)$ with the weak* topology coming from identifying 
$\cP(X)$ as a subset of the unit sphere in the dual of the Banach space of continuous functions on $X$ vanishing at infinity $C_0(X)$. The starting point of this paper is to challenge this and propose a slight modification which will allow to bring into the picture periodic orbits of infinite volume. For that we will shortly concern ourselves with topologizing the space of equivalence classes $[\mu]$ of locally finite measures  $\mu$ on $X$. 

This approach has several advantages over the classical weak* convergence approach. As said above it allows to discuss limiting distributions of infinite volume orbits but also it allows to detect in some cases information which is invisible for the weak* topology: In the classical discussion, it is common that a sequence of periodic probability measures $\mu_{Hx_i}$ converges to 
the zero measure (phenomenon known as full escape of mass). Nevertheless it sometimes happens that the orbits themselves do
converge to a limiting object but this information was lost because the measures along the sequence were not scaled properly. This phenomenon happens for example in~\cite{CasselsEscape} which inspired us to define the notion of convergence to be defined below. 

Although the results we will prove are rather specialized we wish to present the framework in which our discussion takes place 
in some generality. Let $G$ be a Lie group\footnote{One could (and should) develop this discussion in the $S$-arithmetic and 
adelic settings as well.} and let $\Ga<G$ be a lattice.
\begin{question}\label{q}
Let $X=G/\Ga$ and let $H_ix_i$ be a sequence of periodic orbits. Under which conditions the following holds:
\begin{enumerate}
\item The sequence $\br{\mu_{H_ix_i}}$ has a converging subsequence?
\item The accumulation points of $\br{\mu_{H_ix_i}}$ are themselves (homothety classes of) periodic measures?
\end{enumerate}
\end{question}
  
\section{Basic definitions and results}\label{defres}
\subsection{Topologies}
Now we make our discussion in the introduction more rigorous. Let $X$ be a locally compact second countable Hausdorff space and $\mathcal M(X)$ the space of locally finite measures on $X$. We say that two locally finite measures $\mu$ and $\nu$ in $\mathcal M(X)$ are equivalent if there exists a constant $\lambda>0$ such that $\mu=\lambda\nu$. This forms an equivalence
relation and we denote the equivalence class of $\mu$ by $[\mu]$. We denote by $\bP\cM(X)$ the set of all equivalence classes of nonzero locally finite measures on $X$.

We topologize $\mathcal M(X)$ and $\bP\cM(X)$ as follows. Let $C_c(X)$ be the space of compactly supported continuous functions on $X$. For any $\rho\in C_c(X)$, define $$i_\rho:\mathcal M(X)\to C_0(X)^*$$ by sending $d\mu\in\mathcal M(X)$ to $\rho d\mu\in C_0(X)^*$.
Here $C_0(X)$ is the space of continuous functions on $X$ vanishing at infinity equipped with the supremum norm, and $C_0(X)^*$ denotes its dual space. The weak* topology on $C_0(X)^*$ then induces a topology $\tau_\rho$ on $\mathcal M(X)$ via the map $i_\rho$. We will denote by $\tau_X$ the topology on $\mathcal M(X)$ generated by $(\mathcal M(X),\tau_\rho)$ $(\rho\in C_c(X))$. Equivalently,  $\tau_X$ is the smallest topology on $\mathcal M(X)$ such that for any $f\in C_c(X)$ the map $$\mu\mapsto\int fd\mu$$ is a continuous map from $\mathcal M(X)$ to $\mathbb R$.

\begin{definition}\label{def12}
Let $\pi_P$ be the natural projection map from $\mathcal M(X)\setminus\{0\}$ to $\bP\cM(X)$. We define $\tau_P$ to be the quotient topology on $\bP\cM(X)$ induced by $\tau_X$ via $\pi_P$. In other words, $U$ is an open subset in $\bP\cM(X)$ if and only if $\pi_P^{-1}(U)$ is open in $\mathcal M(X)\setminus\{0\}$. In this way, we obtain a topological space $(\bP\cM(X),\tau_P)$.
\end{definition}

\subsection{Main results}
Let $G=\SL(n,\mathbb R)$, $\Gamma=\SL(n,\mathbb Z)$ and $X=G/\Gamma$. Denote by $m_X$ the unique $G$-invariant probability measure on $X$ and by $\Ad$ the adjoint representation of $G$. We write $$A=\{\diag(e^{t_1},e^{t_2},\dots,e^{t_{n-1}},e^{t_n}): t_1+t_2+\dots+t_n=0\}$$ for the connected component of the full diagonal group in $G$, and $$N=\{(u_{ij})_{1\leq i,j\leq n}: u_{ii}=1\;(1\leq i\leq n),\;u_{ij}=0\;(i>j)\}$$ for the upper triangular unipotent group. Let $K=\SO(n,\mathbb R)$. In this paper, we address Question \ref{q} in the space $X=\SL(n,\mathbb R)/\SL(n,\mathbb Z)$ with certain periodic orbits $H_ix_i$, and prove the convergence of $[\mu_{H_ix_i}]$ with respect to the topology $(\tau_P, \bP\cM(X))$. As a simple exercise, and to motivate
such a statement, the reader can show that if $\br{\mu_{H_ix_i}}\to\br{m_X}$ for example, then the orbits $H_ix_i$ become 
dense in $X$. In many cases our results imply that indeed the limit homothety class is the class of the uniform measure $m_X$.

Before stating our theorems, we need some notations. For a Lie subgroup $H<G$, let $H^0$ denote the connected component of identity of $H$, and $\Lie(H)$ its Lie algebra.  Denote by $C_G(H)$ (resp. $  C_G(\Lie(H)))$ the centralizer of $H$ (resp. $\Lie(H))$ in $G$. We write $\mathfrak g=\Lie(G)=\mathfrak{sl}(n,\mathbb R)$, and $$\exp:\mathfrak{sl}(n,\mathbb R)\to\SL(n,\mathbb R)$$ the exponential map from $\mathfrak g$ to $G$. We also write $\|\cdot\|_{\mathfrak g}$ for the norm on $\mathfrak g$ induced by the Euclidean norm on the space of $n\times n$ matrices. For any $g\in G$ and any measure $\mu$ on $X$, define the measure $g_*\mu$ by 
$$g_*\mu(E)=\mu(g^{-1}E)\textup{ for any Borel subset }E\subset X.$$ An $A$-orbit $Ax$ in $X$ is called divergent if the map $a\mapsto ax$ from $A$ to $X$ is proper.

\begin{definition}\label{def13}
Let $\{g_k\}_{k\in\mathbb N}$ be a sequence in $G$. For any subgroup $S\subset A$, we define $$\mathcal A(S,\{g_k\}_{k\in\mathbb N})=\{Y\in\Lie(S): \{\Ad(g_k)Y\}_{k\in\mathbb N}\text{ is bounded in }\mathfrak g\}.$$ This is a subalgebra in $\Lie(S)$. 
\end{definition}
\begin{remark}\label{rem:1054}
By definition \ref{def13}, $\{\Ad(g_k)Y\}_{k\in\mathbb N}$ is unbounded for any $Y\in\Lie(S)\setminus\mathcal A(S,\{g_k\}_{k\in\mathbb N})$. Then one can find a subsequence $\{g_{i_k}\}_{k\in\mathbb N}$ such that for any $Y\in\Lie(S)\setminus\mathcal A(S,\{g_{i_k}\}_{k\in\mathbb N})$, the sequence $\{\Ad(g_{i_k})Y\}_{k\in\mathbb N}$ diverges to infinity. 

Indeed, suppose that for an element $Y\in\Lie(S)\setminus\mathcal A(S,\{g_k\}_{k\in\mathbb N})$, $\{\Ad(g_k)Y\}_{k\in\mathbb N}$ does not diverge. Then there is a subsequence $\{g'_k\}_{k\in\mathbb N}$ such that $\{\Ad(g'_k)Y\}_{k\in\mathbb N}$ is bounded. This implies that $\mathcal A(S,\{g_k'\}_{k\in\mathbb N})$ contains the linear span of $Y$ and $\mathcal A(S,\{g_k\}_{k\in\mathbb N})$. Because of this, one can keep on enlarging the set $\mathcal A(S,\{g_k\}_{k\in\mathbb N})$ by passing to subsequences of $\{g_k\}_{k\in\mathbb N}$. But due to the finite dimension of $\Lie(S)$, this process would stop at some point. Then one can get a subsequence $\{g_{i_k}\}_{k\in\mathbb N}$ such that for any vector $Y\in\Lie(S)\setminus\mathcal A(S,\{g_{i_k}\}_{k\in\mathbb N})$, the sequence $\Ad(g_{i_k})Y\to\infty$.
\end{remark}

The following theorem answers Question \ref{q} for translates of a divergent diagonal orbit in $G/\Gamma$. Moreover, it gives a description of all accumulation points.

\begin{theorem}\label{th11}
Let $Ax$ be a divergent orbit in $X$. Then for any $\{g_k\}_{k\in\mathbb N}$ in $G$, the sequence $[(g_k)_*\mu_{Ax}]$ has a subsequence converging to an equivalence class of a periodic measure on $X$. 

Furthermore, by passing to a subsequence, we assume that for any $Y\in\Lie(A)\setminus\mathcal A(A,\{g_k\}_{k\in\mathbb N})$ the sequence $\{\Ad(g_k)Y\}_{k\in\mathbb N}$ diverges (see Remark~\ref{rem:1054}). Then we have the following description of the limit points of the sequence $[(g_k)_*\mu_{Ax}]$. The subgroup $\exp(\mathcal A(A,\{g_k\}_{k\in\mathbb N}))$ is the connected component of the center of the reductive group $C_G(\mathcal A(A,\{g_k\}_{k\in\mathbb N}))$, and any limit point of the sequence $[(g_k)_*\mu_{Ax}]$ is a translate of the equivalence class $[\mu_{C_G(\mathcal A(A,\{g_k\}_{k\in\mathbb N}))^0x}]$. In particular, if $\mathcal A(A,\{g_k\}_{k\in\mathbb N})=\{0\}$, then $[(g_k)_*\mu_{Ax}]$ converges to the equivalence class of the Haar measure $m_X$ on $X$.\end{theorem}

In fact, we deduce Theorem \ref{th11} as a corollary of the following theorem.

\begin{theorem}\label{nth11}
Let $Ax$ be a divergent orbit in $X$. Suppose that $\{g_k\}_{k\in\mathbb N}$ is a sequence in $N$ with $$g_k=(u_{ij}(k))_{1\leq i,j\leq n}\in\SL(n,\mathbb R)$$ such that for each pair $(i,j)$ $(1\leq i<j\leq n)$,$$\textup{either }u_{ij}(k)=0\textup{ for any }k,\;\textup{ or }u_{ij}(k)\to\infty\textup{ as }k\to\infty.$$ Then the sequence $[(g_k)_*\mu_{Ax}]$ converges to the equivalence class $[\mu_{C_G(\mathcal A(A,\{g_k\}_{k\in\mathbb N}))^0x}]$.
\end{theorem}

We will also deduce the following theorem from Theorem \ref{th11} and Theorem \ref{nth11}, which answers Question \ref{q} for translates of an orbit of a connected reductive group $H$ containing $A$. We will see by Lemma \ref{al101} that for such a reductive group $H$, and for $x\in X$ with $Ax$ divergent, $Hx$ is a closed orbit.

\begin{theorem}\label{th12}
Let $Ax$ be a divergent orbit in $X$ and let $H$ be a connected reductive group containing $A$. Then for any $\{g_k\}_{k\in\mathbb N}$ in $G$, the sequence $[(g_k)_*\mu_{Hx}]$ has a subsequence converging to an equivalence class of a periodic measure on $X$. 

Furthermore, let $S$ be the connected component of the center of $H$, and assume that for any $Y\in\Lie(S)\setminus\mathcal A(S,\{g_k\}_{k\in\mathbb N})$ the sequence $\{\Ad(g_k)Y\}_{k\in\mathbb N}$ diverges. Then we have the following description of the limit points of $[(g_k)_*\mu_{Hx}]$. The subgroup $\exp(\mathcal A(S,\{g_k\}_{k\in\mathbb N}))$ is the connected component of the center of the reductive group $C_G(\mathcal A(S,\{g_k\}_{k\in\mathbb N}))$, and any limit point of the sequence $[(g_k)_*\mu_{Hx}]$ is a translate of the equivalence class $[\mu_{C_G(\mathcal A(S,\{g_k\}_{k\in\mathbb N}))^0x}]$. In particular, if $\mathcal A(S,\{g_k\}_{k\in\mathbb N})=\{0\}$, then $[(g_k)_*\mu_{Hx}]$ converges to the equivalence class of the Haar measure $m_X$ on $X$.
\end{theorem}
\begin{remark}
The proof of Theorem \ref{th11} also gives a criterion on the convergence of $[(g_k)_*\mu_{Ax}]$. Similar criterion on the convergence of $[(g_k)_*\mu_{Hx}]$ for a connected reductive group $H$ containing $A$ could be obtained from the proof of Theorem \ref{th12}.
\end{remark}

We give several examples to illustrate Theorem \ref{th11}, Theorem \ref{nth11} and Theorem \ref{th12}. 
\begin{enumerate}
\item[(1)] Let $G=\SL(3,\mathbb R)$ and $\Gamma=\SL(3,\mathbb Z)$. Pick the initial point $x=\mathbb Z^n\in X$ and the sequence $g_k=\smallmat{1&k&k^2/2\\0&1&k\\0&0&1}$. In this case one can show that the subalgebra $\mathcal A(A,\{g_k\}_{k\in\mathbb N})=\{0\}$, and $\Ad(g_k)Y$ diverges for any nonzero $Y\in\Lie(A)$. We also have $C_G(\mathcal A(A,\{g_k\}_{k\in\mathbb N}))=\SL(3,\mathbb R)$. Theorem \ref{th11} then says that $[(g_k)_*\mu_{Ax}]$ converges to $[\mu_{\SL(3,\mathbb R)x}]=[m_X]$.

\item[(2)] Fix $G,\Gamma,x$ and $g_k$ as in example (1). Let $H$ be the connected component of the reductive subgroup $$\left(\begin{array}{ccc}* & * & 0 \\ * & * & 0 \\0 & 0 & *\end{array}\right)\cap\SL(3,\mathbb R).$$ Then the center $S$ of $H$ is equal to $\left\{\diag(a,a,a^{-2}):a\neq0\right\}$, and one could check that the subalgebra $\mathcal A(S,\{g_k\}_{k\in\mathbb N})=\{0\}$, and $\Ad(g_k)Y$ diverges for any nonzero $Y\in\Lie(S)$. Also $C_G(\mathcal A(S,\{g_k\}_{k\in\mathbb N}))=\SL(3,\mathbb R)$. Then Theorem \ref{th12} implies that the sequence $[(g_k)_*\mu_{Hx}]$ converges to $[\mu_{\SL(3,\mathbb R)x}]=[m_X]$.
\item[(3)] Let $G=\SL(4,\mathbb R)$ and $\Gamma=\SL(4,\mathbb Z)$. Pick the initial point $x=\mathbb Z^n\in X$ and the sequence $g_k=\smallmat{1&k&0&0\\0&1&0&0\\0&0&1&k\\0&0&0&1}$. In this case one can show that $\mathcal A(A,\{g_k\}_{k\in\mathbb N})=\left\{\diag(t,t,-t,-t):t\in\mathbb R\right\}$ and $$C_G(\mathcal A(A,\{g_k\}_{k\in\mathbb N}))=\left(\begin{array}{cccc}* & * & 0 & 0 \\ * & * & 0 & 0 \\0 & 0 & * & * \\0 & 0 & * & *\end{array}\right)\cap\SL(4,\mathbb R).$$ Theorem \ref{nth11} then says that the sequence $[(g_k)_*\mu_{Ax}]$ converges to $[\mu_{C_G(\mathcal A(A,\{g_k\}_{k\in\mathbb N}))^0x}].$
\item[(4)] Fix $G,\Gamma$ and $x$ as in example (3), and pick the sequence $g_k=\smallmat{1&k&k^2/2&0\\0&1&k&0\\0&0&1&0\\0&0&0&1}$. Let $H$ be the connected component of the reductive subgroup $$\left(\begin{array}{cccc}* & * & 0 & 0\\ * & * & 0 & 0\\0 & 0 & * & 0\\0&0&0&*\end{array}\right)\cap\SL(4,\mathbb R).$$ Then the center $S$ of $H$ is equal to $\left\{\diag(a,a,b,c):a^2bc=1\right\}$, and one could check that $\mathcal A(S,\{g_k\}_{k\in\mathbb N})=\left\{\diag(s,s,s,-3s):s\in\mathbb R\right\}$ and $$C_G(\mathcal A(S,\{g_k\}_{k\in\mathbb N}))=\left(\begin{array}{cccc}* & * & * & 0\\ * & * & * & 0\\ * & * & * & 0\\0&0&0&*\end{array}\right)\cap\SL(4,\mathbb R).$$ In this case, Theorem \ref{th12} tells that any limit point of the sequence $[(g_k)_*\mu_{Hx}]$ is a translate $[\mu_{C_G(\mathcal A(S,\{g_k\}_{k\in\mathbb N}))^0x}]$, and the proof of Theorem \ref{th12} would imply that $[(g_k)_*\mu_{Hx}]$ actually converges to $[\mu_{C_G(\mathcal A(S,\{g_k\}_{k\in\mathbb N}))^0x}]$.
\end{enumerate}

By comparing examples (1) and (3) (resp. (2) and (4)), one can see that the subalgebra $\mathcal A(A,\{g_k\}_{k\in\mathbb N})$ (resp. $\mathcal A(S,\{g_k\}_{k\in\mathbb N})$) plays an important role in indicating what kinds of limit points the sequence $[(g_k)_*\mu_{Ax}]$ (resp. $[(g_k)_*\mu_{Hx}]$) could converge to. In example (1), we have $\mathcal A(A,\{g_k\}_{k\in\mathbb N})=\{0\}$. By pushing $Ax$ with $g_k$, the orbit $g_kAx$ starts snaking in the space $\SL(3,\mathbb R)/\SL(3,\mathbb Z)$, and eventually fills up the entire space. In example (3), $\mathcal A(A,\{g_k\}_{k\in\mathbb N})$ is a 1-dimensional subalgebra in $\Lie(A)$ which commutes with $g_k$, and it corresponds to the part of the orbit $Ax$ which stays still and is not affected when we push $\mu_{Ax}$ by $g_k$. This would result in the limit orbit having this part as the `central direction', and the `orthogonal complement' part in $Ax$ would be pushed by $g_k$ and fill up the sub-homogeneous space $\smallmat{\SL(2,\mathbb R) & 0 \\0 & \SL(2,\mathbb R)}x$ in $\SL(4,\mathbb R)/\SL(4,\mathbb Z)$.

By the characterization of convergence given in Proposition \ref{p25}, Theorem \ref{th11} and Theorem \ref{th12} can be restated in the form of the following
\begin{theorem}\label{th13}
Let $Ax$ be a divergent orbit and $\{g_k\}_{k\in\mathbb N}$ be a sequence in $G$ such that $[(g_k)_*\mu_{Ax}]$ converges to an equivalence class of a locally finite periodic measure $[\nu]$ as in Theorem \ref{th11}. Then there exists a sequence $\lambda_k>0$ such that $$\lambda_k(g_k)_*\mu_{Ax}\to\nu$$ with respect to the topology $\tau_X$. In particular, for any $F_1,F_2\in C_c(X)$ we have $$\frac{\int F_2d(g_k)_*\mu_{Ax}}{\int F_1d(g_k)_*\mu_{Ax}}\to\frac{\int F_2d\nu}{\int F_1d\nu}$$ whenever $\int F_1d\nu\neq0$. The same results hold if $A$ is replaced by any connected reductive group $H$ containing $A$.
\end{theorem}
\begin{remark}
From the proofs of Theorem \ref{th11} and Theorem \ref{nth11}, we will see that in the case $\mathcal A(A,\{g_k\}_{k\in\mathbb N})=\{0\}$, the numbers $\lambda_k$'s in Theorem \ref{th13} are related to the volumes of convex polytopes of a special type in $\Lie(A)$ (see Definition \ref{def31} and Corollary \ref{c91}). We remark here that in view of Theorem \ref{th13}, the $\lambda_k$'s in this case can also be calculated by a function $F_1\in C_c(X)$ with its support being a large compact subset. This makes Theorem \ref{th13} practical in other problems.
\end{remark}
\subsection{Applications}
As an application of our results, we give one example of a counting problem. More details about this counting problem could be found in \cite{DRS93}, \cite{EM93}, \cite{EMS96} and \cite{Sha00}. 

Let $M(n,\mathbb R)$ be the space of $n\times n$ matrices with the norm $$\|M\|^2=\operatorname{Tr}(M^tM)=\sum_{1\leq i,j\leq n} x_{ij}^2$$ for $M=(x_{ij})_{1\leq i,j\leq n}\in M(n,\mathbb R)$. Denote by $B_T$ the ball of radius $T$ centered at 0 in $M(n,\mathbb R)$. Fix a monic polynomial $p_0(\lambda)$ in $\mathbb Z[\lambda]$ which splits completely over $\mathbb Q$. By Gauss lemma, the roots $\alpha_i$ of $p(\lambda)$ are integers. We assume that the $\alpha_i$'s are distinct and nonzero. Let $$M_\alpha=\diag(\alpha_1,\alpha_2,\dots,\alpha_n)\in M(n,\mathbb Z).$$ For $M\in M(n,\mathbb R)$, denote by $p_M(\lambda)$ the characteristic polynomial of $M$. 
We define $$V(\mathbb R):=\{M\in M(n,\mathbb R): p_M(\lambda)=p_0(\lambda)\}$$ the variety of matrices $M$ with characteristic polynomial $p_M(\lambda)$ equal to $p_0(\lambda)$, and $$V(\mathbb Z):=\{M\in M(n,\mathbb Z): p_M(\lambda)=p_0(\lambda)\}$$ the integer points in the variety $V(\mathbb R)$.

The metric $\|\cdot\|_{\mathfrak g}$ on $\mathfrak g=\mathfrak{sl}_n(\mathbb R)$ induces Haar measures on $A$ and $N$. The $K$-invariant probability measure on $K$ and the Haar measures on $A$, $N$ then give a Haar measure on $G$ via Iwasawa Decomposition $G=KNA$. We will denote by $c_X$ the volume of $X=G/\Gamma$ with respect to the Haar measure on $G$.

 There is a natural volume form on the variety $V(\mathbb R)$ inherited from $G=\SL(n,\mathbb R)$. Specifically, the orbit map $$G\to V(\mathbb R)$$ defined by $g\mapsto\Ad(g)M_\alpha$ gives an isomorphism between the quotient space $G/C_G(A)$ and the variety $V(\mathbb R)$, and the volume form is defined to be the $G$-invariant measure on $G/C_G(A)$. The existence of such a measure is well-known, and the proof of it could be found, for example, in \cite{Rag}. With this volume form, one can compute (see Proposition \ref{ap101}) that for any $T$, the volume of $V(\mathbb R)\cap B_T$ equals $cT^{n(n-1)/2}$ for some constant $c>0$. The following theorem concerns the asymptotic formula for the number of integer points in $V(\mathbb Z)\cap B_T$. We will see that the set $V(\mathbb Z)\cap B_T$ behaves differently from $V(\mathbb R)\cap B_T$, with an extra log term.

By a well-known theorem of Borel and Harish-Chandra \cite{BHC62}, the subset $V(\mathbb Z)$ is a finite  disjoint union of $\Ad(\Gamma)$-orbits. One can write this disjoint union as $$V(\mathbb Z)=\bigcup_{i=1}^{h_0}\Ad(\Gamma)M_i$$ for some $h_0\in\mathbb N$ and $M_i\in V(\mathbb Z)$ ($1\leq i\leq h_0$). Note that for each $M_i$, the stabilizer $\Gamma_{M_i}$ of $M_i$ is finite. Also the number of the orbits $h_0$ is equal to the number of equivalence classes of nonsingular ideals in the subring in $M(n,\mathbb R)$ generated by $M_\alpha$, for which readers may refer to \cite{BHC62} and \cite{LM33}. In the following theorem, to ease the notation, we write $\mb{t}$ for a vector $(t_1,t_2,\dots,t_n)\in\mathbb R^n$.

\begin{theorem}\label{th14}
We have $$|V(\mathbb Z)\cap B_T|\sim\left(\sum_{i=1}^{h_0}\frac1{|\Gamma_{M_i}|}\right) \frac{c_0\Vol(B_1)}{c_X\prod_{j>i}|\alpha_j-\alpha_i|}T^{n(n-1)/2}(\ln T)^{n-1}$$ where $\Vol(B_1)$ is the volume of the unit ball in $\mathbb R^{n(n-1)/2}$ and $c_0$ is the volume of the $(n-1)$-convex polytope 
$$\left\{\mb{t}\in\mathbb R^n:\sum_{i=1}^n t_i=0, \sum_{j=1}^l t_{i_j}\geq\sum_{j=1}^l(j-i_j),\forall 1\leq l\leq n,\forall 1\leq i_1<\cdots<i_l\leq n\right\}$$ with respect to the natural measure induced by the Lebesgue measure on $\mathbb R^n$.
\end{theorem}

In the sequel, we will mainly focus on Theorem \ref{nth11} as all the other theorems will be corollaries of it. In the course of the proof of Theorem \ref{nth11}, the case $\mathcal A(A,\{g_k\}_{k\in\mathbb N})=\{0\}$ plays an important role, and other cases could be proved by induction. Therefore, most of our arguments in this paper would work for the case $\mathcal A(A,\{g_k\}_{k\in\mathbb N})=\{0\}$. We remark that our proof is inspired by \cite{OS14}, where Hee Oh and Nimish Shah deal with the case $G=\SL(2,\mathbb R)$ by applying exponential mixing and obtain an error estimate. This effective result is improved recently in \cite{KK} by Dubi Kelmer and Alex Kontorovich.

When we showed an earlier draft of the manuscript to Nimish Shah, he pointed out to us that similar results to those appearing in this paper were established by him at the beginning of this century, but were never published.

 The paper is organized as follows:
\begin{itemize}
\item We start our work in section \ref{top} by studying the topology $\tau_P$ on $\bP\cM(X)$ for a locally compact second countable Hausdorff space $X$. In particular, a characterization of convergence in $\bP\cM(X)$ is given, and Theorem \ref{th13} is obtained as a natural corollary, if Theorem \ref{th11} and Theorem \ref{th12} are presumed.
\item In section \ref{con}, a special type of convex polytopes in $\Lie(A)$ is introduced. Such convex polytopes are related to non-divergence of the orbits $g_kAx$. In order to analyze these convex polytopes in the setting of Theorem \ref{nth11}, we define graphs associated to them and prove some auxiliary results concerning the graphs in section \ref{graph}. With the assumption $\mathcal A(A,\{g_k\}_{k\in\mathbb N})=\{0\}$, these auxiliary results imply some properties of the convex polytopes, which are proved in section \ref{recon}.

\item Keeping the assumption $\mathcal A(A,\{g_k\}_{k\in\mathbb N})=\{0\}$ in section \ref{nondiv}, we prove a statement on the non-divergence of the sequence of $[(g_k)_*\mu_{Ax}]$ and show that $[(g_k)_*\mu_{Ax}]$ converges to $[\nu]$ for some probability measure $\nu$ invariant under a unipotent subgroup. Then we translate section \ref{nondiv} in terms of adjoint representation in section \ref{ad-nondiv}. The linearization technique and the measure classification theorem for unipotent actions on homogeneous spaces are discussed in section \ref{line}, which enable us to study the measure rigidity in our setting.
\item We complete the proof of Theorem \ref{nth11} in section \ref{proof}. Then we prove Theorem \ref{th11} and \ref{th12}. The  proof of Theorem \ref{th14} is given in section \ref{app}.
\end{itemize}

\begin{acknowledgments}
We would like to express our gratitude to Nimish Shah for insightful comments and support on this work. We are grateful to Barak Weiss for valuable communications and Roy Meshulam for teaching us Lemma \ref{p32}. We thank Ofir David, Asaf Katz, Rene R\"uhr, Oliver Sargent, Lei Yang, Pengyu Yang and Runlin Zhang for helpful discussions and support. We also thank the anonymous referees for carefully reading our manuscript and their valuable comments. Finally, the authors acknowledge the support of ISF grants number 871/17, 662/15 and 357/13, and the second author is in part supported at the Technion by a Fine Fellowship.
\end{acknowledgments}

\section{Topology on $\bP\cM(X)$}\label{top}
In this section, we study the topology $\tau_P$ on $\bP\cM(X)$ for any locally compact second countable Hausdorff space $X$. We will give a description of the convergence of a sequence $[\mu_k]$ in $\mathbb P\mathcal M(X)$ (Proposition \ref{p25}). This will help us study the convergence of the sequence $[(g_k)_*\mu_{Ax}]$ in Theorems \ref{th11} and \ref{nth11} (resp. $[(g_k)_*\mu_{Hx}]$ in Theorem \ref{th12}).

Before proving Proposition \ref{p25}, we need some preparations.

\begin{proposition}\label{p23}
The topology $(\tau_P, \bP\cM(X))$ is Hausdorff. In particular, any convergent sequence in $\bP\cM(X)$ has a unique limit.
\end{proposition}
\begin{proof}
Let $[\mu]$ and $[\nu]$ be two distinct elements in $\bP\cM(X)$. We choose $f\in C_c(X)$ and representatives $\mu$ and $\nu$ such that $$\int fd\mu=\int fd\nu=1.$$ Since $[\mu]\neq[\nu]$, there exists a nonnegative function $g\in C_c(X)$ such that $$\int gd\mu\neq 1,\quad\int gd\nu=1.$$ We define neighborhoods of $\mu$ and $\nu$ in $\mathcal M(X)$ by $$V(\mu; f,g,\epsilon)=\left\{\lambda:\left|\int gd\lambda-\int gd\mu\right|<\epsilon, \left|\int fd\lambda-\int fd\mu\right|<\epsilon\right\}$$
$$V(\nu; f,g,\epsilon)=\left\{\lambda:\left|\int gd\lambda-\int gd\nu\right|<\epsilon, \left|\int fd\lambda-\int fd\nu\right|<\epsilon\right\}.$$ Since $\pi_P:\mathcal M(X)\setminus\{0\}\to\mathbb P\mathcal M(X)$ is an open map, $\pi_P(V(\mu; f,g,\epsilon))$ and $\pi_P(V(\nu; f,g,\epsilon))$ are open neighborhoods of $[\mu]$ and $[\nu]$ in $\bP\cM(X)$ for any $\epsilon>0$. Let $\kappa=\int gd\mu$. We prove that for any $\epsilon<\min\{0.1,|\kappa-1|/5\}$ $$\pi_P(V(\mu; f,g,\epsilon))\cap\pi_P(V(\nu; f,g,\epsilon))=\emptyset.$$ Suppose, on the contrary, that $[\lambda]\in\pi_P(V(\mu; f,g,\epsilon))\cap\pi_P(V(\nu; f,g,\epsilon))$. Then there exist constants $\alpha, \beta>0$ such that $$\left|\alpha\int gd\lambda-\int gd\mu\right|<\epsilon, \left|\alpha\int fd\lambda-\int fd\mu\right|<\epsilon$$ $$\left|\beta\int gd\lambda-\int gd\nu\right|<\epsilon, \left|\beta\int fd\lambda-\int fd\nu\right|<\epsilon.$$ This implies that $$\frac{\int gd\mu-\epsilon}{\int gd\nu+\epsilon}<\frac\alpha\beta<\frac{\int gd\mu+\epsilon}{\int gd\nu-\epsilon},\quad\frac{\int fd\mu-\epsilon}{\int fd\nu+\epsilon}<\frac\alpha\beta<\frac{\int fd\mu+\epsilon}{\int fd\nu-\epsilon}$$ and $$\frac{\kappa-\epsilon}{1+\epsilon}<\frac\alpha\beta<\frac{\kappa+\epsilon}{1-\epsilon},\quad\frac{1-\epsilon}{1+\epsilon}<\frac\alpha\beta<\frac{1+\epsilon}{1-\epsilon}.$$ This is a contradiction for $\epsilon<\min\{0.1,|\kappa-1|/5\}$.
\end{proof}

\begin{proposition}\label{p24}
A sequence $[\mu_k]$ in $\bP\cM(X)$ converges to $[\nu]$ if and only if for each $k\in\mathbb N$ there exists a representative $\mu_k'$ in $[\mu_k]$ and for $[\nu]$ a representative $\nu'\in[\nu]$ such that $\mu_k'$ converges to $\nu'$ in $\mathcal M(X)$.\end{proposition}
\begin{proof}
Let $[\mu_k]$ be a sequence in $\mathbb P\mathcal M(X)$ converging to $[\nu]$. We choose $f\in C_c(X)$ and representatives $\mu_k'$ and $\nu'$ of $[\mu_k]$ and $[\nu]$ such that $$\int fd\mu_k'=\int fd\nu'=1.$$ Suppose that $\mu_k'\not\to\nu'$ in $\mathcal M(X)$. Then there exists a nonnegative function $g\in C_c(X)$ such that after passing to a subsequence $$\int gd\nu'=1,\quad\left|\int gd\mu_k'-1\right|\geq\delta$$ for some $\delta>0$. Then by the same argument as in Proposition \ref{p23}, we can find a neighborhood $\pi_P(V(\nu;f,g,\epsilon))$ of $[\nu]$ in $\bP\cM(X)$ for some $\epsilon<\min\{0.1,\delta/5\}$ such that $$[\mu_k]\notin \pi_P(V(\nu;f,g,\epsilon))$$ which contradicts the condition $[\mu_k]\to[\nu]$. The other direction follows from Definition \ref{def12}.
\end{proof}

Now we prove the following important proposition, which provides a characterization of the convergence of a sequence $[\mu_k]$ in $\mathbb P\mathcal M(X)$. This will help us study the convergence of equivalence classes of locally finite measures on $\SL(n,\mathbb R)/\SL(n,\mathbb Z)$ in the rest of the paper.

\begin{proposition}\label{p25}
\begin{enumerate}
\item Let $\{\mu_k\}_{k\in\mathbb N}$ be a sequence in $\mathcal M(X)$. Then $[\mu_k]$ converges to $[\nu]$ in $\bP\cM(X)$ if and only if there exists a sequence $\{\lambda_k\}$ of positive numbers such that $\lambda_k\mu_k$ converges to $\nu$ in $\mathcal M(X)$. If there exists another sequence $\{\lambda_k'\}$ with $\lambda_k'\mu_k\to\nu'\neq0$ in $\mathcal M(X)$, then $$[\nu']=[\nu]$$ and $\lim_k\lambda_k'/\lambda_k$ exists.
\item The sequence $[\mu_k]$ converges to $[\nu]$ if and only if for any $f,g\in C_c(X)$ with $\int gd\nu\neq0$, we have $\int gd\mu_k\neq0$ for sufficiently large $k$ and $$\frac{\int fd\mu_k}{\int gd\mu_k}\to\frac{\int fd\nu}{\int gd\nu}.$$
\end{enumerate}
\end{proposition}
\begin{proof}
The first statement follows from Proposition \ref{p23} and Proposition \ref{p24}. For $\lim_k\lambda_k'/\lambda_k$, we choose $f\in C_c(X)$ with $\int fd\nu\neq0$, and we have $$\frac{\lambda_k'}{\lambda_k}=\frac{\lambda_k'\int fd\mu_k}{\lambda_k\int fd\mu_k}\to\frac{\int fd\nu'}{\int fd\nu}.$$

For the second statement, if $[\mu_k]\to[\nu]$, then there exists a sequence $\lambda_k>0$ such that $\lambda_k\mu_k\to\nu\neq0$. For any $f,g\in C_c(X)$ with $\int gd\nu\neq0$ we have $$\lambda_k\int gd\mu_k\neq0$$ for sufficiently large $k$ and $$\frac{\int fd\mu_k}{\int gd\mu_k}=\frac{\int fd(\lambda_k\mu_k)}{\int gd(\lambda_k\mu_k)}\to\frac{\int fd\nu}{\int gd\nu}.$$ Conversely, let $g\in C_c(X)$ with $\int gd\nu\neq0$ and $$\lambda_k=\frac{\int gd\nu}{\int gd\mu_k}.$$ Then we have $\lambda_k\mu_k\to\nu$ and $[\mu_k]\to[\nu]$.
\end{proof}
\begin{remark}
This proves that Theorem \ref{th13} is equivalent to Theorem \ref{th11} and Theorem \ref{th12}.
\end{remark}

From the discussions in this section, we know that to prove Theorem \ref{nth11}, one needs to find a sequence of $\lambda_k>0$ such that $\lambda_k(g_k)_*\mu_{Ax}$ converges to a locally finite measure $\nu$, and then prove that $\nu$ is a periodic measure. From section \ref{con} to section \ref{recon}, we will construct the sequence $\lambda_k$ in an explicit way. In the rest of the paper, $X$ will denote the homogeneous space $G/\Gamma$.

\section{Convex polytopes}\label{con}
In this section, we will construct a special type of convex polytopes in $\Lie(A)$. These convex polytopes will play an important role in the rest of the paper.

By Theorem 1.4 in \cite{TW03}, $Ax$ is divergent in $X=G/\Gamma$ if and only if $x\in A\cdot \SL(n,\mathbb Q)\Gamma$. Note that for any $q\in \SL(n,\mathbb Q)$ the lattice $q\Gamma q^{-1}$ is commensurable with $\Gamma$, and all results in this paper would hold if $\Gamma$ is replaced by $q\Gamma q^{-1}$. Therefore, without loss of generality, we may assume that the initial point $x=x_e=e\SL(n,\mathbb Z)$, where $e$ is the identity matrix in $G$. We will denote by $m_{\Lie(A)}$ the natural measure on $\Lie(A)\subset\mathfrak{sl}(n,\mathbb R)$ induced by the Lebesgue measure on the space of $n\times n$ matrices.

To ease the notations, we will write $\mb{t}$ for a vector $(t_1,t_2,\dots,t_n)$ in a $n$-dimensional space, and denote by $[n]$ the index set $\{1,2,\dots,n\}$. We write $\mathcal I_n$ for the collection of all multi-index subsets of $[n]$, and $\mathcal I_n^l$ for the collection of the index subsets of cardinality $l$ in $\mathcal I_n$. Let $\{e_1,e_2,\dots,e_n\}$ be the standard basis of $\mathbb R^n$. For any index subset $I=\{i_1<i_2\cdots<i_l\}\in\mathcal I_n$, we denote by $$e_{I}:=e_{i_1}\wedge\dots\wedge e_{i_l}$$ the wedge product of the vectors $e_{i_1},\dots,e_{i_l}$. We write $\omega_{I}(\mb{t})$ ($\mb{t}=(t_1,t_2,\dots,t_n)\in\mathbb R^n$) for the linear functional $\sum_{i\in I}t_i$ on $\mathbb R^n$.

Let $g\in\SL(n,\mathbb R)$ and $\delta>0$. We define a region $\Omega_{g,\delta}$ in $\Lie(A)$ as follows. Let $\mb{t}=(t_1,t_2,\dots,t_n)\in\Lie(A)$. For each $e_i\in\mathbb R^n$, the vector $$g\exp(\mb{t})e_i=e^{t_i}ge_i\notin B_\delta$$ if and only if $$t_i\geq\ln\delta-\ln\|ge_i\|.$$ Here $B_\delta$ denotes the ball of radius $\delta>0$ around 0 in $\mathbb R^n$ with the standard Euclidean norm $\|\cdot\|$. We also consider the wedge product $e_{I}$ for any nonempty subset $I\in\mathcal I_n^l$ $(1\leq l\leq n)$, and $$g\exp(\mb{t})e_{I}=e^{\omega_{I}(\mb{t})}g e_{I}\notin B_\delta$$ if and only if $$\omega_{I}(\mb{t})\geq\ln\delta-\ln\|ge_{I}\|.$$ Here by abusing notations, $\|\cdot\|$ is the norm on $\wedge^l\mathbb R^n$ induced by the Euclidean norm on $\mathbb R^n$, and $B_\delta$ is the ball of radius $\delta>0$ around 0 in $\wedge^l\mathbb R^n$. This leads to the following
\begin{definition}\label{def31}
For any $g\in G$ and $\delta>0$, we define $$\Omega_{g,\delta}=\left\{\mb{t}\in\Lie(A): \omega_{I}(\mb{t})\geq\ln\delta-\ln\|g e_{I}\|\text{ for any nonempty }I\in\mathcal I_n\right\}.$$
\end{definition}
 \begin{remark}
By the construction above, for any $\mb{t}\in\Lie(A)\setminus\Omega_{g,\delta}$, the lattice $g\exp(\mb{t})\mathbb Z^n$ has a short nonzero vector with the length depending on $\delta>0$. Hence by Mahler's compactness criterion, the point $g\exp(\mb{t})\Gamma\in gA\Gamma$ is close to infinity. Due to this reason, we will mainly study the part $\{g\exp(\mb{t})\Gamma:\mb{t}\in\Omega_{g,\delta}\}$ of the orbit $gA\Gamma$.
\end{remark}

\begin{lemma}\label{p31}
The region $\Omega_{g,\delta}$ is a bounded convex polytope in $\Lie(A)$ for any $g\in G$ and $\delta>0$.
\end{lemma}
\begin{proof}
Since the region $\Omega_{g,\delta}$ is defined by various linear functionals on $\Lie(A)$, $\Omega_{g,\delta}$ is a convex polytope. Now by definition, $\Omega_{g,\delta}$ is contained in the following region $$\left\{\mb{t}\in\mathbb R^n: \sum_{i=1}^n t_i=0, t_i\geq\ln\delta-\ln\|ge_i\|,\forall i\in[n]\right\}$$ which is bounded. The boundedness of $\Omega_{g,\delta}$ then follows. 
\end{proof}

In section \ref{recon}, we will closely study the convex polytope $\Omega_{g,\delta}$. We list here some properties of convex polytopes which will be used later. For a bounded convex subset $\Omega$ in a Euclidean space $E$, we denote by $\Vol(\Omega)$ the volume of $\Omega$ with respect to the Lebesgue measure on $E$, and by $\Area(\partial\Omega)$ the surface area of the boundary $\partial\Omega$ of $\Omega$ induced by the Lebesgue measure. 

The following lemma is well known. We learnt it from Roy Meshulam.
\begin{lemma}\label{p32}
Let $\Omega$ be a bounded convex subset in $\mathbb R^d$. Suppose that $\Omega$ contains a ball of radius $r>0$. Then we have $$\frac{\Area(\partial\Omega)}{\Vol(\Omega)}\leq\frac dr.$$
\end{lemma}
\begin{proof}
Let $B_{r}(0)$ denote the ball of radius $r$ centered at $0$ in $\mathbb R^d$ and we may assume, without loss of generality, that $B_r(0)\subset\Omega$. We have
\begin{eqnarray*}
\Area(\partial\Omega)&=&\lim_{\epsilon\to0}\frac{\Vol(\Omega+\epsilon B_1(0))-\Vol(\Omega)}\epsilon\\
&=&\lim_{\epsilon\to0}\frac{\Vol(\Omega+(\epsilon/r) B_r(0))-\Vol(\Omega)}\epsilon\\
&\leq&\lim_{\epsilon\to0}\frac{\Vol(\Omega+(\epsilon/r)\Omega)-\Vol(\Omega)}\epsilon\\
&=&\lim_{\epsilon\to0}\frac{(1+(\epsilon/r))^d-1}\epsilon\Vol(\Omega)=\frac dr\Vol(\Omega).
\end{eqnarray*}
This completes the proof of the lemma.
\end{proof}

\begin{lemma}\label{p33}
Let $R\subset\Omega$ be two bounded $d$-dimensional convex polytopes in $\mathbb R^d$. Suppose that $\Omega$ contains a ball of radius $r>0$ and $$\frac{\Vol(R)}{\Vol(\Omega)}\geq c$$ for some constant $c>0$. Then $R$ contains a ball of radius $rc/d$.
\end{lemma}
\begin{proof}
Let $\rho$ be the largest number such that $R$ contains a ball of radius $\rho$. It suffices to show that $\rho\geq rc/d$. Let $\{f_i\}$ be the collection of the facets of $R$, and denote by $P_i$ the hyperplane determined by $f_i$. First, we prove two claims.

Claim 1: Let $p$ be a point in $R$, and let $f_{i_0}$ be a facet of $R$ such that the hyperplane $P_{i_0}$ is closest to $p$ among all the hyperplanes $P_i$. Then the orthogonal projection of $p$ in $P_{i_0}$ is in the facet $f_{i_0}$.
\begin{proof}[Proof of Claim 1]
Let $p_{i_0}$ be the orthogonal projection of $p$ in $P_{i_0}$, and denote by $\overline{p_{i_0}p}$ the line segment connecting $p$ and $p_{i_0}$. Suppose that $p_{i_0}$ is outside the facet $f_{i_0}$. Then $\overline{p_{i_0}p}$ intersects another facet of $R$, say, $f_{j_0}$. This implies that the distance between $p$ and the hyperplane $P_{j_0}$ is smaller than the length of $\overline{p_{i_0}p}$, which contradicts the choice of $P_{i_0}$.
\end{proof}

Claim 2: $\operatorname{Vol}(R)\leq\rho\Area(\partial R).$
\begin{proof}[Proof of Claim 2]
 For each facet $f_i$ of $R$, let $B_i$ be the unique cylinder with the following properties:
\begin{enumerate}
\item the base of $B_i$ is $f_i$, and the height of $B_i$ is equal to $\rho$.
\item $B_i$ and $R$ lie in the same half-space determined by $P_i$.
\end{enumerate}
The maximality of $\rho$ then implies $$R\subset\bigcup_i B_i;$$ otherwise, by Claim 1, one would find a point $x\in R\setminus\bigcup_iB_i $ such that for each $f_i$, the distance between $x$ and $f_i$ is strictly larger than $\rho$. Now we have $$\operatorname{Vol}(R)\leq\sum_i\operatorname{Vol}(B_i)=\rho\Area(\partial R)$$ and Claim 2 follows.
\end{proof}

Now we can finish the proof of the lemma. By Claim 2 and Lemma \ref{p32}, we have $$\rho\geq\frac{\operatorname{Vol}(R)}{\Area(\partial R)}\geq\frac{c\operatorname{Vol}(\Omega)}{\Area(\partial \Omega)}\geq\frac{cr}{d}.$$ Here we use the fact that $\Area(\partial R)\leq\Area(\partial\Omega)$ for any two convex polytopes $R\subset\Omega$.
\end{proof}

For a bounded convex polytope $\Omega$ in $\mathbb R^d$ and $\epsilon>0$, its $\epsilon$-neighborhood is defined by $$\{\mb{t}\in\mathbb R^d: \inf_{\mb{s}\in\Omega}\|\mb{t}-\mb{s}\|\leq\epsilon\}.$$ Here $\|\cdot\|$ is the Euclidean norm on $\mathbb R^d$.

\begin{lemma}\label{p34}
Let $\Omega$ be a bounded convex subset in $\mathbb R^d$ which contains a ball of radius $r>0$. Let $\Omega_\epsilon$ be the $\epsilon$-neighborhood of $\Omega$ for $\epsilon>0$. Then we have $$\frac{\Vol(\Omega_\epsilon)}{\Vol(\Omega)}\leq\left(1+\frac\epsilon r\right)^d.$$
\end{lemma}
\begin{proof}
The proof is similar to Lemma \ref{p32}. Assume that $\Omega$ contains the ball $B_r(0)$ of radius $r$ around $0$. We have
\begin{eqnarray*}
\frac{\Vol(\Omega_\epsilon)}{\Vol(\Omega)}&=&\frac{\Vol(\Omega+(\epsilon/r) B_r(0))}{\Vol(\Omega)}\leq\frac{\Vol(\Omega+(\epsilon/r)\Omega)}{\Vol(\Omega)}=\left(1+\frac\epsilon r\right)^d.
\end{eqnarray*}
This completes the proof of the lemma.
\end{proof}

\section{Auxiliary results in graph theory}\label{graph}
In this section, we will study a special class of graphs and prove some properties of these graphs (Proposition \ref{p41} and Lemma \ref{l41}), which will be crucial in our study of convex polytopes in section \ref{recon}. From now on until section \ref{proof}, we would assume that $\{g_k\}_{k\in\mathbb N}$ satisfies the condition in Theorem \ref{nth11}, i.e. $\{g_k\}_{k\in\mathbb N}$ is a sequence in $N$ with $$g_k=(u_{ij}(k))_{1\leq i,j\leq n}$$ such that for each $(i,j)$ $(1\leq i<j\leq n)$, either $u_{ij}(k)=0\textup{ for any }k$, or $u_{ij}(k)\neq0$ and diverges to infinity as $k\to\infty$.

In order to prove Proposition \ref{p41}, we will need some lemmas involving complex calculations which will guarantee the validity of the proof of Proposition \ref{p41}. Here we introduce the following notation. For any $g\in\SL(n,\mathbb R)$ and any $1\leq l\leq n$, denote by $(g)_{l\times l}$ the $l\times l$ submatrix in the upper left corner of $g$. Note that if $g,h\in\SL(n,\mathbb R)$ are upper triangular, then $(gh)_{l\times l}=(g)_{l\times l}(h)_{l\times l}$.

\begin{lemma}\label{al51}
For any $a\in A$ and any $1\leq l\leq n$, we have either $(g_k)_{l\times l}=(a^{-1}g_ka)_{l\times l}$ for all $k$ or $(g_k)_{l\times l}\neq(a^{-1}g_ka)_{l\times l}$ for all $k$.
\end{lemma}
\begin{proof}
Write $a=(a_1,a_2,\dots,a_n)\in A$. By definition, we have $$(g_k)_{l\times l}=(u_{ij}(k))_{1\leq i,j\leq l}$$ and $$(a^{-1}g_ka)_{l\times l}=(a_i^{-1}a_ju_{ij}(k))_{1\leq i,j\leq l}.$$ The equation $(g_k)_{l\times l}=(a^{-1}g_ka)_{l\times l}$ then yields $$\text{either }u_{ij}(k)=0\text{ or } a_i=a_j,\quad\forall1\leq i,j\leq l.$$ Now the lemma follows from the dichotomy assumption on the entries of $g_k$ $(k\in\mathbb N)$.
\end{proof}

\begin{lemma}\label{al52}
Let $a\in A$. Suppose that the sequence $\{g_kag_k^{-1}\}_{k\in\mathbb N}$ is bounded in $\SL(n,\mathbb R)$. Then $g_k$ commutes with $a$ for any $k$. 
\end{lemma}
\begin{proof}
Suppose not. Then by Lemma \ref{al51} with $l=n$ we have $$g_k\neq a^{-1}g_ka,\;\forall k\in\mathbb N.$$ In this case, we would like to find a contradiction. 

 Let $l_0$ be the minimum of the integers $0\leq l\leq n-1$ with the property $$(g_k)_{(l+1)\times (l+1)}\neq(a^{-1}g_ka)_{(l+1)\times (l+1)}$$ for any $k$. By Lemma \ref{al51}, $l_0$ is also the maximum of $0\leq l\leq n-1$ such that $(g_k)_{l\times l}$ commutes with $(a)_{l\times l}$ for all $k$.

We write $a=\diag(a_1,a_2,\dots,a_n)\in A$. Then for any $1\leq l\leq n$ $$(a)_{l\times l}=\diag(a_1,a_2,\dots,a_l).$$ We also write $$(g_k)_{(l_0+1)\times(l_0+1)}=\left(\begin{array}{cc}(g_k)_{l_0\times l_0} & \mb{v}_k \\0 & 1\end{array}\right)\in\SL(l_0+1,\mathbb R)$$ where $\mb{v}_k$ is the $l_0$-dimensional column vector next to $(g_k)_{l_0\times l_0}$ in $g_k$. Since $(g_k)_{l_0\times l_0}$ commutes with $(a)_{l_0\times l_0}$, one can compute 
\begin{align*}
(a^{-1}g_ka)_{(l_0+1)\times(l_0+1)}=&(a^{-1})_{(l_0+1)\times(l_0+1)}(g_k)_{(l_0+1)\times(l_0+1)}(a)_{(l_0+1)\times(l_0+1)}\\
=&\left(\begin{array}{cc}(g_k)_{l_0\times l_0} & a_{l_0+1}(a^{-1})_{l_0\times l_0}\mb{v}_k \\0 & 1\end{array}\right)=\left(\begin{array}{cc}(g_k)_{l_0\times l_0} & \mb{w}_k \\0 & 1\end{array}\right)
\end{align*} 
where $$\mb{w}_k:=a_{l_0+1}(a^{-1})_{l_0\times l_0}\mb{v}_k.$$ As $(g_k)_{(l_0+1)\times(l_0+1)}$ does not commute with $(a)_{(l_0+1)\times(l_0+1)}$, we have $$\mb{v}_k\neq\mb{w}_k.$$ From this and the dichotomy assumption on the entries of $g_k$ $(k\in\mathbb N)$, one can then deduce that  $\mb{v}_k\neq\mb{0}$, $\mb{v}_k\to\infty$ and $$\mb{w}_k-\mb{v}_k=\left(a_{l_0+1}(a^{-1})_{l_0\times l_0}-\operatorname{I}_{l_0}\right)\mb{v}_k\to\infty$$ as $k\to\infty$. Here $\operatorname{I}_{l_0}$ is the $l_0\times l_0$ identity matrix.

Now one can compute
\begin{eqnarray*}
(a^{-1}g_kag_k^{-1})_{(l_0+1)\times(l_0+1)}&=&(a^{-1}g_ka)_{(l_0+1)\times(l_0+1)}(g_k^{-1})_{(l_0+1)\times(l_0+1)}\\
&=&\left(\begin{array}{cc}(g_k)_{l_0\times l_0} & \mb{w_k} \\0 & 1\end{array}\right)\left(\begin{array}{cc}(g_k)_{l_0\times l_0} & \mb{v}_k \\0 & 1\end{array}\right)^{-1}\\
&=&\left(\begin{array}{cc}\operatorname{I}_{l_0} & \mb{w}_k-\mb{v}_k \\0 & 1\end{array}\right).
\end{eqnarray*}
Since $\mb{w}_k-\mb{v}_k\to\infty$ as $k\to\infty$, the equation above implies that $\{a^{-1}g_kag_k^{-1}\}_{k\in\mathbb N}$ diverges, which contradicts the boundedness of $\{g_kag_k^{-1}\}_{k\in\mathbb N}$. This completes the proof of the lemma.
\end{proof}

\begin{corollary}\label{c42}
Let $S\subset A$ be a subgroup in $A$. Then for any $\mb{t}\in\Lie(S)$, either $\Ad(g_k)\mb{t}\to\infty$ as $k\to\infty$ or $\Ad(g_k)\mb{t}=\mb{t}$ for all $k$.
\end{corollary}
\begin{proof}
Apply Lemma \ref{al51} and Lemma \ref{al52} with $a=\exp(\mb{t})$.
\end{proof}

\begin{definition}\label{def41}
We define a graph $G(\{g_k\}_{k\in\mathbb N})=(V,E)$ associated to $\{g_k\}_{k\in\mathbb N}$ as follows. The set of vertices $V$ is the index set $[n]=\{1,2,\dots,n\}$. Two vertices $i<j$ are connected by an edge in the edge set $E$, which we denote by $i\sim j$, if $u_{ij}(k)\to\infty$ as $k\to\infty$.
\end{definition}

Now we can prove our first result in this section.

\begin{proposition}\label{p41}
The subalgebra $\mathcal A(A,\{g_k\}_{k\in\mathbb N})$ of $\Lie(A)$ (as defined in Definition \ref{def13}) is trivial if and only if the graph $G(\{g_k\}_{k\in\mathbb N})$ associated to $\{g_k\}_{k\in\mathbb N}$ is connected.
\end{proposition}
\begin{proof}
Suppose that the graph $G(\{g_k\}_{k\in\mathbb N})$ associated to $\{g_k\}_{k\in\mathbb N}$ is not connected. Let $G_l=(V_l,E_l)$ $(1\leq l\leq m)$ be the connected components of $G(\{g_k\}_{k\in\mathbb N})$. We pick $x_l\in\mathbb R\setminus\{0\}$ such that $\sum_{l=1}^m |V_l|x_l=0$. Now if a vertex $i\in V_l\subset[n]$, we set $t_i=x_l$. In this way we obtain an element $\mb{t}=(t_i)_{1\leq i\leq n}\in\Lie(A)\setminus\{0\}$. Note that $\mb{t}$ is invertible. We show that $$g_k\mb{t}=\mb{t}g_k.$$ Indeed, since $\mb{t}$ is invertible, we compute $$\mb{t}g_k\mb{t}^{-1}=(t_it_j^{-1}u_{ij}(k))_{1\leq i,j\leq n}.$$ For $u_{ij}(k)\neq0$, by the definition of the graph $G(\{g_k\}_{k\in\mathbb N})$, the vertices $i$ and $j$ are in the same connected component. Hence we have $t_i=t_j$ and $$\mb{t}g_k\mb{t}^{-1}=(t_it_j^{-1}u_{ij}(k))_{1\leq i,j\leq n}=(u_{ij}(k))_{1\leq i,j\leq n}=g_k$$ as desired. This implies that $\Ad(g_k)$ fixes $\mb{t}$, and by definition $\mb{t}\in\mathcal A(A,\{g_k\}_{k\in\mathbb N})\neq\{0\}$.

Now assume that the graph $G(\{g_k\}_{k\in\mathbb N})$ is connected. Suppose that $\mathcal A(A,\{g_k\}_{k\in\mathbb N})$ is not zero. Then there exists an element $\mb{t}\in\Lie A\setminus\{0\}$ such that $\Ad(g_k)\mb{t}$ is bounded as $k\to\infty$. Let $a=\exp \mb{t}\in A\setminus\{e\}$. Then $\{g_kag_k^{-1}\}$  is bounded in $\SL(n,\mathbb R)$. By Lemma \ref{al52}, $g_k$ commutes with $a$. If we write $a=\diag(a_1,a_2,\dots,a_n)$, then the equation $g_k=ag_ka^{-1}$ yields $$(u_{ij}(k))_{1\leq i,j\leq n}=(a_ia_j^{-1}u_{ij}(k))_{1\leq i,j\leq n}$$ and hence $a_i=a_j$ whenever $u_{ij}(k)\neq0$. The connectedness of the graph $G(\{g_k\}_{k\in\mathbb N})$ then implies that all $a_i$'s are equal and $a=e$, which contradicts $a\in A\setminus\{e\}$. This completes the proof of the proposition.
\end{proof}

\begin{definition}\label{def42}
Let $G(V,E)$ be a graph consisting of the set of vertices $V$ and the set of edges $E$. Here we assume $V=\{v_1,v_2,\dots,v_n\}$ is an ordered set with the ordering $\prec$, and we denote by $v_i\sim v_j$ if $v_i$ and $v_j$ are connected by an edge in $E$. A subset $S\subset V$ is called UDS (uniquely determined by successors) if it satisfies the following property: for any $v_i\in V$ 
\begin{equation}\label{eq}
v_i\in S\implies v_j\in S\text{ for all $j\prec i$ with $v_j\sim v_i$}
\end{equation}
\end{definition}

For our purpose, we will consider UDS subsets of $[n]$ in the graph $G(\{g_k\}_{k\in\mathbb N})$ associated to $\{g_k\}_{k\in\mathbb N}$. The ordering of $[n]$ inherits the natural ordering on $\mathbb N$. The following proposition will be needed in our computations later.

\begin{proposition}\label{p43}
For any $1\leq l\leq n$ and any nonempty $I\in\mathcal I_n^l$, the sequence $\{g_ke_{I}\}_{k\in\mathbb N}\subset\wedge^l\mathbb R^n$ is bounded if and only if $I$ is UDS in the vertex set $[n]$ of $G(\{g_k\}_{k\in\mathbb N})$. If this case happens, then we have $g_k e_{I}=e_{I}$ for any $k\in\mathbb N$.
\end{proposition}
\begin{proof}
Let $I=\{i_1<i_2<\cdots<i_l\}$. Suppose that $\{g_ke_{I}\}_{k\in\mathbb N}$ is bounded. We show that $I$ is UDS in $[n]$. If not, let $i_0$ be the minimum in $I=\{i_1,\dots,i_l\}$ such that the property (\ref{eq}) in Definition \ref{def42} does not hold for $i_0$. Then there is $j_0<i_0$ with $j_0\sim i_0$ but $j_0\notin I$. By the minimality of $i_0$, for any $i\in I=\{i_1,i_2,\dots,i_l\}$ with $j_0<i<i_0$, we have $j_0\not\sim i$; otherwise $j_0\in I$.  This implies that $u_{j_0,i}(k)=0$ for all $i\in\{i_1,i_2,\dots,i_l\}$ with $i<i_0$. Note that $u_{j_0,i_0}(k)\to\infty$ as $k\to\infty$ by our assumption on the entries of $g_k$ $(k\in\mathbb N)$. 

Now we compute $g_ke_{I}$. In particular, by expanding $g_ke_{I}$ in terms of the standard basis $\{e_{J}:J\in\mathcal I_n^l\}$ in $\wedge^l\mathbb R^n$, we are interested in the coefficient in the $e_{J_0}$-coordinate, where $J_0=\{i\in I:i\neq i_0\}\cup\{j_0\}$. As $u_{j_0,i}(k)=0$ for all $i\in\{i_1,i_2,\dots,i_l\}$ with $i<i_0$, one can compute $$g_ke_{I}=u_{j_0,i_0}(k)(\wedge_{i\in I,i<i_0}e_i)\wedge e_{j_0}\wedge(\wedge_{i\in I,i>i_0}e_i)+\sum_{J\neq J_0}c_{J}e_{J}$$ for some $c_{J}\in\mathbb R$ $(J\neq J_0)$.  The divergence of $u_{j_0,i_0}(k)$ then contradicts the boundedness of $g_ke_{I}$. This proves that $I$ is UDS.

Conversely, suppose that $I$ is a UDS subset in $[n]$. In this case, we will show inductively that for any $1\leq j\leq l$  $$g_k(e_{i_1}\wedge e_{i_2}\wedge\dots\wedge e_{i_j})=e_{i_1}\wedge e_{i_2}\wedge\dots\wedge e_{i_j}$$ and hence obtain that $g_ke_{I}=g_k(e_{i_1}\wedge e_{i_2}\wedge\dots\wedge e_{i_l})$ remains fixed. For $j=1$, since $\{i_1,\dots,i_l\}$ is UDS, this implies that $u_{i,i_1}=0$ for all $i<i_1$ and $g_ke_{i_1}=e_{i_1}$. Now assume that the formula holds for $j$. For $j+1$, we know that $$g_ke_{i_{j+1}}=e_{i_{j+1}}+\sum_{i\in\{i_1,\dots,i_j\}} u_{i,i_{j+1}}(k)e_i$$ and hence 
\begin{eqnarray*}
g_k(e_{i_1}\wedge e_{i_2}\wedge\dots\wedge e_{i_j}\wedge e_{i_{j+1}})&=&e_{i_1}\wedge e_{i_2}\wedge\dots\wedge e_{i_j}\wedge(g_ke_{i_{j+1}})\\
&=&e_{i_1}\wedge e_{i_2}\wedge\dots\wedge e_{i_j}\wedge e_{i_{j+1}}.
\end{eqnarray*}
This completes the proof of the proposition.
\end{proof}

Finally, we will show the following lemma, which will be crucial in our study of convex polytopes in section \ref{recon}.

\begin{lemma}\label{l41}
Let $G(V,E)$ be a connected graph, where $V=\{v_1,v_2,\dots,v_n\}$ is an ordered set with the ordering $\prec$. Then we can assign values $x_1,x_2,\dots,x_n$ to the vertices $v_1,v_2,\dots,v_n$ such that 
\begin{enumerate}
\item $\sum_{v_i\in V}x_i=0$
\item For any proper UDS subset $S\subset V$, $\sum_{v_i\in S}x_i>0.$
\end{enumerate}
\end{lemma}
\begin{proof}
We use induction on the number of vertices in $G(V,E)$. There is nothing to prove for $n=1$. Now suppose that we have $n+1$ vertices. Assume without loss of generality that $v_1$ is the smallest according to the ordering $\prec$ on $V$. We remove the vertex $v_1$ and all the edges adjacent to $v_1$ from the graph $G$. This yields a new graph $G'$ with $m$ connected components $G'_1=(V'_1,E'_1),\dots,G'_m=(V'_m,E'_m)$ for some $m\in\mathbb N$. Since $|V'_j|\leq n$ $(1\leq j\leq m)$ and $V'_j$ inherits the ordering from $V$, we can apply the induction hypothesis on each $G_j'=(V'_j,E'_j)$. In particular, we obtain a vector $(x'_2,\dots,x'_{n+1})\in\mathbb R^n$ such that the value assignment $$v_i\mapsto x'_i,\quad2\leq i\leq n+1$$ satisfies conditions (1) and (2) for each of the graphs $G'_j$ $(1\leq j\leq m)$.

Now we pick a sufficiently small positive number $\epsilon>0$ such that the new value assignment $x_i=x_i'-\epsilon$ $(2\leq i\leq n+1)$ still satisfies condition (2) for each $G'_j=(V'_j,E'_j)$, and let $x_1=n\epsilon$. We show that this value assignment $$v_i\mapsto x_i,\quad1\leq i\leq n+1$$ meets our requirements for $G(V,E)$. The sum of $x_i$ is zero by induction hypothesis. For a proper UDS subset $S\subset V$, if $v_1\notin S$, then $$S=\bigcup_{j=1}^m S'_j$$ where $S'_j$ is a subset in $G'_j=(V'_j,E'_j)$ $(1\leq j\leq m)$, and either $S'_j$ is a proper UDS subset in $G'_j=(V'_j,E'_j)$ or $S'_j=V'_j$. Since $v_1\notin S$, by the connectedness of $G(V,E)$ and the UDS property of $S$, there is some $j$ with $S'_j\neq V'_j$  and hence by taking $\epsilon$ sufficiently small, $$\sum_{v_i\in S}x_i=\sum_{j=1}^m\sum_{v_i\in S_j'} x_i>0.$$ If $S=\{v_1\}$, then condition (2) holds automatically. If $v_1\in S$ and $S\neq\{v_1\}$, then $$S\setminus\{v_1\}=\bigcup_{j=1}^m S'_j$$ where $S'_j$ is a subset in $G'_j=(V'_j,E'_j)$ $(1\leq j\leq m)$, and either $S'_j$ is a proper UDS subset in $G'_j=(V'_j,E'_j)$ or $S'_j=V'_j$. Since $S$ is proper in $V$, there is some $j$ with $S'_j\neq V'_j$ and hence we have $$\sum_{v_i\in S}x_i=\sum_{j=1}^m\sum_{v_i\in S'_j}x_i+x_1>(-n\epsilon)+n\epsilon=0.$$ This completes the proof of the lemma.
\end{proof}

\section{Revisit convex polytopes}\label{recon}
In this section, we will study the convex polytopes $\Omega_{g_k,\delta}$, where $\{g_k\}_{k\in\mathbb N}$ is a sequence in $G$ satisfying the condition in Theorem \ref{nth11}. Our aim in this section is Proposition \ref{p53}, which shows a crucial property of $\Omega_{g_k,\delta}$ about its surface area and its volume. This property will play an important role in various places of the paper.

In the proof of Theorem \ref{nth11}, the case of $\mathcal A(A,\{g_k\}_{k\in\mathbb N})=\{0\}$ plays a central role, and other cases can be deduced from this case. We remark here that in view of Corollary \ref{c42}, $\mathcal A(A,\{g_k\}_{k\in\mathbb N})=\{0\}$ if and only if the limit points of $\{\Ad(g_k)\Lie(A)\}_{k\in\mathbb N}$ in the Grassmanian manifold of $\mathfrak g$ are subalgebras consisting of nilpotent matrices. So starting from this section to section \ref{line}, we will make additional assumptions on $\{g_k\}_{k\in\mathbb N}$ that $\mathcal A(A,\{g_k\}_{k\in\mathbb N})=\{0\}$, and by passing to a subsequence, $\Ad(g_k)\Lie(A)$ converges to a subalgebra consisting of nilpotent matrices in the Grassmanian manifold of $\mathfrak g$. We write $\lim_{k\to\infty}\Ad(g_k)\Lie(A)$ for the limiting subalgebra and $\lim_{k\to\infty}\Ad(g_k)A$ for the corresponding limiting unipotent subgroup.

\begin{lemma}\label{p51}
For any $0<\delta<1$, the region $$\{\mb{t}\in\Lie(A): \omega_{I}(\mb{t})\geq\ln\delta,\forall\text{ nonempty proper UDS }I\in\mathcal I_n\}$$ is a convex subset in $\Lie(A)$ which contains an unbound open cone.
\end{lemma}
\begin{proof}
It suffices to prove the lemma for the region $$\{\mb{t}\in\Lie(A): \omega_{I}(\mb{t})\geq0,\forall\text{ nonempty proper UDS }I\in\mathcal I_n\}.$$ By our assumptions on $\{g_k\}_{k\in\mathbb N}$ and Proposition \ref{p41}, the graph $G(\{g_k\}_{k\in\mathbb N})$ associated to $\{g_k\}_{k\in\mathbb N}$ is connected. Now by applying Lemma \ref{l41} with the graph $G(\{g_k\}_{k\in\mathbb N})$, one can find $\mb{x}=(x_1,x_2,\dots,x_n)\in\Lie(A)$ such that $$\mb{x}\in\{\mb{t}\in\Lie(A): \omega_{I}(\mb{t})>0,\forall\text{ nonempty proper UDS }I\in\mathcal I_n\}.$$ Then by linearity, for any $\lambda>0$ $$\lambda\mb{x}\in\{\mb{t}\in\Lie(A): \omega_{I}(\mb{t})>0,\forall\text{ nonempty proper UDS }I\in\mathcal I_n\}.$$
This implies that there exists an unbounded open cone around the axis $\{\lambda\mb{x},\lambda>0\}$, which is contained in $$\{\mb{t}\in\Lie(A): \omega_{I}(\mb{t})\geq0,\forall\text{ nonempty proper UDS }I\in\mathcal I_n\}.$$ This completes the proof of the lemma.
\end{proof}

\begin{lemma}\label{p52}
Let $0<\delta<1$. For every $k\in\mathbb N$, the region $\Omega_{g_k,\delta}$ contains a ball $B_k$ of radius $r_k$, and $r_k\to\infty$ as $k\to\infty$.
\end{lemma}
\begin{proof}
By definition, we know that
$$\Omega_{g_k,\delta}=\bigcap_{I\in\mathcal I_n}\left\{\mb{t}\in\Lie(A):\omega_{I}(\mb{t})\geq\ln\delta-\ln\|g_ke_{I}\|\right\}.$$ Note that the origin belongs to $\Omega_{g_k,\delta}$, because each $g_k$ is in the upper triangular unipotent subgroup and $\|g_ke_I\|\geq1$ for nonempty $I\in\mathcal I_n$. Now we can write 
\begin{align*}
\Omega_{g_k,\delta}&=\bigcap_{I\text{ UDS }}\left\{\omega_{I}(\mb{t})\geq\ln(\delta/\|g_ke_{I}\|)\right\}\cap\bigcap_{I\text{ non-UDS }}\left\{\omega_{I}(\mb{t})\geq\ln(\delta/\|g_ke_{I}\|)\right\}\\
&=\bigcap_{I\text{ UDS }}\left\{\omega_{I}(\mb{t})\geq\ln\delta\right\}\cap\bigcap_{I\text{ non-UDS }}\left\{\omega_{I}(\mb{t})\geq\ln(\delta/\|g_ke_{I}\|)\right\}
\end{align*}
where we use $g_ke_{I}=e_{I}$ for any UDS set $I$ by Proposition \ref{p43}. For a non-UDS set $I$, we have $g_ke_{I}\to\infty$ as $k\to\infty$.

Since $g_ke_{I}\to\infty$ for any non-UDS set $I$, the region $$\bigcap_{I\text{ non-UDS }}\left\{\omega_{I}(\mb{t})\geq\ln(\delta/\|g_ke_{I}\|)\right\}$$ contains a large ball $S_k$ around the origin for sufficiently large $k$. By Lemma \ref{p51}, the region $$\bigcap_{I\text{ UDS }}\left\{\omega_{I}(\mb{t})\geq\ln\delta\right\}$$ contains an unbounded cone $C$ (which does not depend on $k$) with cusp at the origin. This implies that $$\Omega_{g_k,\delta}\supset S_k\cap C$$ and $\Omega_{g_k,\delta}$ contains a large ball $B_k$ of radius $r_k$ with $r_k\to\infty$ as $k\to\infty$.
\end{proof}

\begin{proposition}\label{p53}
For any $0<\delta<1$, we have $$\lim_{k\to\infty}\frac{\Area(\partial\Omega_{g_k,\delta})}{\operatorname{Vol}(\Omega_{g_k,\delta})}=0.$$
\end{proposition}
\begin{proof}
The proposition follows from Lemma \ref{p32} and Lemma \ref{p52}.
\end{proof}

Actually, we will apply the following variant of Proposition \ref{p53} later.
\begin{corollary}\label{ac61}
Let $0<\delta_1<\delta_2<1$. Then $$\lim_{k\to\infty}\frac{\operatorname{Vol}(\Omega_{g_k,\delta_2})}{\operatorname{Vol}(\Omega_{g_k,\delta_1})}=1.$$
\end{corollary}
\begin{proof}
By definition, we know that $\Omega_{g_k,\delta_2}\subset\Omega_{g_k,\delta_1}$. Let $\{f_i\}$ be the collection of the facets of $\Omega_{g_k,\delta_1}$, and denote by $P_i$ the hyperplane determined by $f_i$. For each $f_i$, let $B_i$ be the unique cylinder with the following properties:
\begin{enumerate}
\item the base of $B_i$ is $f_i$, and the height of $B_i$ is equal to $\ln\delta_2-\ln\delta_1$.
\item $B_i$ and $\Omega_{g_k,\delta_1}$ lie in the same half-space determined by $P_i$.
\end{enumerate}
Then one has $$\Omega_{g_k,\delta_1}\subset \bigcup_i B_i\cup\Omega_{g_k,\delta_2}$$ and $$\operatorname{Vol}(\Omega_{g_k,\delta_1})\leq\sum_i\operatorname{Vol}(B_i)+\operatorname{Vol}(\Omega_{g_k,\delta_2})=(\ln\delta_2-\ln\delta_1)\Area(\partial\Omega_{g_k,\delta_1})+\operatorname{Vol}(\Omega_{g_k,\delta_2})$$ Now the corollary follows from Proposition \ref{p53}.
\end{proof}

From now on, we will fix a $\delta>0$ for any $g\in G$ in the notation $\Omega_{g,\delta}$ unless otherwise specified. For each $k\in\mathbb N$, we choose the representative $$\frac1{\Vol(\Omega_{g_k,\delta})}(g_k)_*\mu_{Ax_e}$$ in $[(g_k)_*\mu_{Ax_e}]$. We will show in the following section that these representatives converge to a locally finite measure $\nu$. We will denote by $$\mu_{Ax_e}|_{\Omega_{g_k,\delta}}$$ the restriction of $\mu_{Ax_e}$ on $\exp({\Omega_{g_k,\delta}})x_e$.

\section{Nondivergence}\label{nondiv}
In this section, we will study the nondivergence of the sequence $$\frac1{\Vol(\Omega_{g_k,\delta})}(g_k)_*\mu_{Ax_e}.$$ The study relies on a growth property of a special class of functions studied by Eskin, Mozes and Shah \cite{EMS97}, and a non-divergence theorem proved by Kleinbock and Margulis~\cite{KM98,Kle10}. As a corollary we will deduce that these measures actually converge to a probability measure, which is invariant under a unipotent subgroup. This is where Ratner's theorem will come into play in section \ref{line} and help us prove the measure rigidity. The goal in this section is to prove Proposition \ref{p65}.

First, we need the following definition of a class of functions, which is introduced in \cite{EMS97}.

\begin{definition}[{\cite[Definition 2.1]{EMS97}}]\label{def71}
Let $d\in\mathbb N$ and $\lambda>0$ be given. Define by $E(d,\lambda)$ the set of functions $f:\mathbb R\to\mathbb C$ of the form $$f(t)=\sum_{i=1}^da_ie^{\lambda_it}\quad(\forall t\in\mathbb R)$$ where $a_i\in\mathbb C$ and $\lambda_i\in\mathbb C$ with $|\lambda_i|\leq\lambda$.
\end{definition}

The following proposition describes the growth property of functions in $E(d,\lambda)$. We denote by $m_{\mathbb R}$ the Lebesgue measure on $\mathbb R$.

\begin{proposition}[{\cite[Corollary 2.10]{EMS97}}]\label{p61}
For any $d\in\mathbb N$ and $\lambda>0$, there exists a constant $\delta_0=\delta_0(d,\lambda)$ satisfying the following: for any $\epsilon>0$, there exists $M>0$ such that for any $f\in E(d,\lambda)$ and any interval $\Xi$ of length at most $\delta_0$ 
\begin{equation}\label{e}
m_{\mathbb R}(\{t\in \Xi: |f(t)|<(1/M)\sup_{t\in \Xi}|f(t)|\})\leq\epsilon m_{\mathbb R}(\Xi).
\end{equation}
\end{proposition}

For any nonzero discrete subgroup $\Lambda$ in $\mathbb R^{n}$, one could define its co-volume as follows. Let $\{v_1,v_2,\dots,v_l\}$ be a $\mathbb Z$-basis of $\Lambda$, where $l$ is the rank of $\Lambda$. Then the co-volume of $\Lambda$ is defined to be the length of the wedge product $v_1\wedge\dots\wedge v_l$ in $\wedge^l\mathbb R^n$, where the norm in $\wedge^l\mathbb R^n$ is induced by the Euclidean norm on $\mathbb R^n$. By abusing notations, we will write $\|\Lambda\|$ for the co-volume of $\Lambda$. One could check that this notion of co-volume is well defined.

The following theorem is essentially proved in \cite{Kle10} and \cite{KM98}. 

\begin{theorem}[Cf. {\cite[Theorem 3.4]{Kle10}, \cite[Theorem 5.2]{KM98}}]\label{th61}
Let $d\in\mathbb N$ and $\lambda>0$. Let $\delta_0=\delta_0(d,\lambda)$ be as in Proposition \ref{p61}. Suppose that an interval $\Xi\subset\mathbb R$ of length at most $\delta_0$, $0<\rho<1$ and a continuous map $h: \Xi\to\SL(n,\mathbb R)$ are given. Assume that for any nonzero discrete subgroup $\Delta$ in $\mathbb Z^n$ we have
\begin{enumerate}
\item the function $x\to\|h(x)\Delta\|^2$ on $\Xi$ belongs to $E(d,\lambda)$ and
\item $\sup_{x\in \Xi}\|h(x)\Delta\|\geq\rho$.
\end{enumerate}
Then for any $\epsilon<\rho$, there exists a constant $\delta(\epsilon)>0$ depending only on $d$ and $\lambda$ such that $$m_{\mathbb R}(\{x\in \Xi:h(x)\mathbb Z^n\cap B_{\delta(\epsilon)}\neq\{0\}\})\leq\epsilon m_{\mathbb R}(\Xi).$$
\end{theorem}
\begin{proof}
The proof is the same as in \cite[Theorem 5.2]{KM98}, but the inequality (\ref{e}) is used instead of the $(C,\alpha)$-good property.
\end{proof}

\begin{lemma}\label{p62}
Let $E$ be a normed vector space, and let $\alpha_i$ $(1\leq i\leq m)$ be different linear functionals on $E$. Then for any $r>0$, we can find $m$ vectors $x_1,x_2,\dots,x_m\in B_r(0)$ such that $$\det\left(\left(e^{\alpha_i(x_j)}\right)_{1\leq i, j\leq m}\right)\neq0.$$ Here $B_r(0)$ is the ball of radius $r$ around $0$ in $E$.
\end{lemma}
\begin{proof}
We can find a line $L$ through the origin such that $\alpha_i|L$ are different functionals defined on $L$. This could be achieved by picking a line which avoids all the kernels of $\alpha_i-\alpha_j$. Hence it suffices to prove the lemma for $\dim E=1$.

Let $E=\mathbb R$ and $\alpha_i(x)=\lambda_ix$ for different $\lambda_i$'s. We will show inductively that for any $r>0$ there exist $x_1,x_2,\dots,x_m\in(-r,r)$ such that $$\det\left(\left(e^{\lambda_ix_j}\right)_{1\leq i,j\leq m}\right)\neq0.$$ It is easy to verify for $m=1$. Now for $m+1$ different $\lambda_i$'s, we compute 
\begin{eqnarray*}
\det\left(\left(e^{\lambda_ix_j}\right)_{1\leq i, j\leq m+1}\right)=e^{\lambda_1x_{m+1}}A_1+e^{\lambda_2x_{m+1}}A_2+\dots+e^{\lambda_{m+1}x_{m+1}}A_{m+1}
\end{eqnarray*}
 where $A_{m+1}=\det\left(\left(e^{\lambda_ix_j}\right)_{1\leq i, j\leq m}\right)$. By induction hypothesis, we can find $x_1,x_2,\dots,x_m\in(-r,r)$ such that $A_{m+1}\neq0$. By the fact that $e^{\lambda_ix}$ $(1\leq i\leq m)$ are linearly independent functions, and by the choice of $x_1,x_2,\dots,x_m$, the function $\det\left(\left(e^{\lambda_ix_j}\right)_{1\leq i, j\leq m+1}\right)$ is a nonzero analytic function in $x_{m+1}$. Since zeros of any analytic function are isolated, this implies that there exists $x_{m+1}\in(-r,r)$ such that $\det\left(\left(e^{\lambda_ix_j}\right)_{1\leq i, j\leq m+1}\right)\neq0$.
\end{proof}

The following proposition describes the supremum of a special function. We will need this proposition to verify the assumption (ii) in Theorem \ref{th61}.

\begin{proposition}\label{p63}
Let $E$ and $V$ be normed vector spaces, and $v_i\in V$ $(1\leq i\leq m)$. Let $f$ be a map from $E$ to $V$ defined by $$f(x)=\sum_{i=1}^m e^{\alpha_i(x)}v_i$$ where $\alpha_i$'s $(1\leq i\leq m)$ are different linear functionals on $E$. Suppose that on an open ball $R\subset E$ of radius $r>0$ we have $$e^{\alpha_i(x)}\|v_i\|\geq M,\quad\forall x\in R,\;1\leq i\leq m$$ for some $M>0$. Then there exists a constant $c>0$ which only depends on the $\alpha_i$'s and $r$ such that $$\sup_{x\in R}\|f(x)\|\geq cM.$$
\end{proposition}
\begin{proof}
Let $x_0$ be the center of $R$ and $B_r(0)$ the ball of radius $r$ around $0$ in $E$. Then $R=x_0+B_r(0)$. By Lemma \ref{p62},  we can find $y_j\in B_r(0)$ $(1\leq j\leq m)$ such that $$\det\left(\left(e^{\alpha_i(y_j)}\right)_{1\leq j,i\leq m}\right)\neq0.$$ We fix this choice of $y_j$'s which only depends on $\alpha_i$'s and $r$. Let $x_j=x_0+y_j\in R$ $(1\leq j\leq m)$. We have $$\left(e^{\alpha_i(y_j)}\right)_{1\leq j,i\leq m}\left(e^{\alpha_i(x_0)}v_i\right)_{1\leq i\leq m}=\left(f(x_j)\right)_{1\leq j\leq m}$$ $$\left(e^{\alpha_i(x_0)}v_i\right)_{1\leq i\leq m}=\left(e^{\alpha_i(y_j)}\right)_{1\leq j,i\leq m}^{-1}\left(f(x_j)\right)_{1\leq j\leq m}.$$ Let $C$ be the matrix norm of $\left(e^{\alpha_i(y_j)}\right)_{1\leq j,i\leq m}^{-1}$. Since $$e^{\alpha_i(x_0)}\|v_i\|\geq M\quad(1\leq i\leq m),$$ this implies that one of $\|f(x_j)\|$ $(1\leq j\leq m)$ is at least $M/mC$. Hence $\sup_{x\in R}\|f(x)\|\geq cM$ with $c=1/mC$.
\end{proof}

For any $g\in G$, $x_0\in\Lie(A)$, a unit vector $\vec v\in\Lie(A)$ and $w=\sum_{I\in\mathcal I_n^l}w_Ie_I\in\wedge^l\mathbb R^n$ ($w_I\in\mathbb R$), the function $$t\mapsto\|g\exp(x_0+t\vec{v})\cdot w\|^2$$ belongs to $E(d,\lambda)$, where $d=n^{2l}$, $\lambda=2l$ and $\|\cdot\|$ is the norm on $\wedge^l\mathbb R^n$ induced by the Euclidean norm on $\mathbb R^n$. Indeed, $$\exp(x_0+t\vec{v})\cdot w=\sum_{I\in\mathcal I_n^l}w_I\exp(x_0+t\vec{v})\cdot e_I$$ is a vector in $\wedge^l\mathbb R^n$ with coordinates being exponential functions of $t$. Hence $g\exp(x_0+t\vec{v})\cdot w$ is a vector whose coordinates are sums of exponential functions of $t$. By a simple calculation, one could get that the function $\|g\exp(x_0+t\vec{v})\cdot w\|^2$ belongs to $E(d,\lambda)$ with $d=n^{2l}$ and $\lambda=2l$. In what follows, we will study functions of this kind. 

With the help of Theorem \ref{th61} and Proposition \ref{p63}, we can now study the nondivergence of the sequence $\frac1{\Vol(\Omega_{g_k,\delta})}(g_k)_*\mu_{Ax}.$ We write $$\mathcal K_r:=\{g\Gamma\in G/\Gamma:\text{ every nonzero vector in $g\mathbb Z^n$ has norm $\geq r$}\}.$$ By Mahler's compactness criterion, this is a compact subset in $G/\Gamma$. The following proposition is crucial in the proof of Proposition \ref{p65}.

\begin{proposition}\label{p64}
For any $\epsilon>0$, there exists a constant $\delta(\epsilon)>0$ such that for sufficiently large $k\in\mathbb N$ $$m_{\Lie(A)}(\{\mb{t}\in\Omega_{g_k,\delta}: g_k\exp(\mb{t})\mathbb Z^n\notin\mathcal K_{\delta(\epsilon)}\})\leq\epsilon m_{\Lie(A)}(\Omega_{g_k,\delta}).$$
\end{proposition}
\begin{proof}
Fix a unit vector $\vec v\in\Lie(A)$ such that  the values in $$\{\omega_{I}(\vec v): I\in\mathcal I_n\}$$ are all different. Let $d=n^{2n}$ and $\lambda=2n$ such that for any $x_0\in\Lie(A)$, $l\in\mathbb N$, $w\in\wedge^l\mathbb R^n$ and $k\in\mathbb N$, the function $$\|g_k\exp(x_0+t\vec{v})\cdot w\|^2,\quad t\in\mathbb R$$ belongs to $E(d,\lambda)$ as defined in Definition \ref{def71}. We will write $\delta_0$ for the constant $\delta_0(d,\lambda)$ defined in Proposition \ref{p61}.

We can find a cover of $\Omega_{g_k,\delta}$ by countably many disjoint small boxes of diameter at most $\delta_0$ such that each box is of the form $$B=\{x_0+t\vec{v}: x_0\in S,\; t\in \Xi\}$$ where $S$ is the base of $B$ perpendicular to $\vec v$ and $\Xi=[0,\delta_0]$. We denote by $\mathcal F$ the collection of these boxes. Let $\mathcal F=\mathcal F_1\cup\mathcal F_2$ where $\mathcal F_1$ is the collection of the boxes in $\mathcal F$ which intersect $\partial\Omega_{g_k,\delta}$ and $\mathcal F_2=\mathcal F\setminus\mathcal F_1$. Then for any box $B\in\mathcal F_2$, $B$ is contained in $\Omega_{g_k,\delta}$.

Since the diameter of each box in $\mathcal F$ is at most $\delta_0$, in view of Lemma \ref{p34}, Lemma \ref{p52} and Proposition \ref{p53}, for any $\epsilon>0$, we have $$m_{\Lie(A)}\left(\bigcup_{B\in\mathcal F_1}B\right)\leq\frac\epsilon2 m_{\Lie(A)}(\Omega_{g_k,\delta})$$ for sufficiently large $k$. In order to prove the proposition, it suffices to show that for any $\epsilon>0$, there exists $\delta(\epsilon)>0$ such that for each box $B\in\mathcal F_2$, we have $$m_{\Lie(A)}(\{\mb{t}\in B: g_k\exp(\mb{t})\mathbb Z^n\notin\mathcal K_{\delta(\epsilon)}\})\leq\frac\epsilon2 m_{\Lie(A)}(B).$$
Now fix a box $B\in\mathcal F_2$ with $$B=\{x_0+t\vec{v}: x_0\in S,\; t\in \Xi\}$$ where $S$ is the base of $B$ and $\Xi=[0,\delta_0]$. We will apply Theorem \ref{th61}. Let $\Delta$ be a nonzero discrete subgroup of rank $l$ in $\mathbb Z^n$ with a $\mathbb Z$-basis $\{v_1,v_2,\dots,v_l\}\subset\mathbb Z^n.$ The wedge product $v_1\wedge\dots\wedge v_l\in\wedge^l\mathbb R^n$ can be written as $$v_1\wedge\dots\wedge v_l=\sum_{I\in\mathcal I_n^l}a_{I}e_{I}$$ where $a_{I}\in\mathbb Z$. We define a map from $B$ to $\wedge^l\mathbb R^n$ by 
\begin{eqnarray*}
f_\Delta(\mb{t})&=&(g_k\exp\mb{t})(v_1\wedge\cdots\wedge v_l)=\sum_{I\in\mathcal I_n^l}a_{I}e^{\omega_{I}(\mb{t})}g_ke_{I},\quad\mb{t}\in B.
\end{eqnarray*}
For each $x_0\in S$, we consider the map $$t\mapsto f_\Delta(x_0+t\vec{v})$$ from $\Xi=[0,\delta_0]$ to $\wedge^l\mathbb R^n$.
Since $B\subset\Omega_{g_k,\delta}$, by our construction of $\Omega_{g_k,\delta}$, we have $$\|e^{\omega_{I}(x_0+t\vec{v})}g_ke_{I}\|\geq\delta,\quad\forall t\in\Xi,\;\forall I\in\mathcal I_n^l.$$  By Proposition \ref{p63}, we have $$\sup_{t\in \Xi}\|f_\Delta(x_0+t\vec{v})\|\geq c\delta.$$ Note that by Proposition \ref{p63}, this inequality holds with a uniform constant $c>0$ depending only on $\omega_{I}(\mb{t})$ $(I\in\mathcal I_n)$ and $\delta_0$ for any nonzero $\Delta\subset\mathbb Z^n$. Since $\|f_\Delta(x_0+t\vec{v})\|$ is the co-volume of $g_k(\exp(x_0+t\vec{v}))\Delta$ and $\|f_\Delta(x_0+t\vec{v})\|^2$ is a function in $E(d,\lambda)$, we can apply Theorem \ref{th61} and obtain that $$m_{\mathbb R}(\{t\in \Xi:g_k\exp(x_0+t\vec{v})\mathbb Z^n\notin\mathcal  K_{\delta(\epsilon)}\})\leq\frac\epsilon2 m_{\mathbb R}(\Xi)$$ for some constant $\delta(\epsilon)>0$ and for any $x_0\in S$. Now by integrating the inequality above over the region $x_0\in S$, we have $$m_{\Lie(A)}(\{\mb{t}\in B: g_k\exp(\mb{t})\mathbb Z^n\notin\mathcal K_{\delta(\epsilon)}\})\leq\frac\epsilon2 m_{\Lie(A)}(B).$$ The proposition now follows.
\end{proof}

Now we can prove the main result in this section.
\begin{proposition}\label{p65}
By passing to a subsequence, the sequence $\frac1{\Vol(\Omega_{g_k,\delta})}(g_k)_*(\mu_{Ax_e}|_{\Omega_{g_k,\delta}})$ converges to a probability measure $\nu$. Furthermore, we have $$\frac1{\Vol(\Omega_{g_k,\delta})}(g_k)_*\mu_{Ax_e}\to\nu$$ and hence the sequence $[(g_k)_*\mu_{Ax_e}]$ converges to $[\nu]$. Here the probability measure $\nu$ is invariant under the action of the unipotent subgroup $\lim_{n\to\infty}\Ad(g_k)A$.
\end{proposition}
\begin{proof}
Suppose that the sequence of probability measures $$\mu_k:=\frac1{\Vol(\Omega_{g_k,\delta})}(g_k)_*(\mu_{Ax_e}|_{\Omega_{g_k,\delta}})$$ weakly converges to a measure $\nu$ after passing to a subsequence. We show that $\nu$ is a probability measure. It is obvious that $\nu(X)\leq1$.

Now for any $\epsilon>0$, let $\mathcal K_{\delta(\epsilon)}$ be the compact subset in $G/\Gamma$ as in Proposition \ref{p64}. Let $f_\epsilon$ be a nonnegative continuous function with compact support on $G/\Gamma$ such that $0\leq f_\epsilon\leq1$ and $f_\epsilon=1$ on $\mathcal K_{\delta(\epsilon)}$. Then we have 
\begin{align*}
\nu(X)\geq\int_Xf_\epsilon d\nu=\lim_{k\to\infty}\int_Xf_\epsilon d\mu_k\geq\limsup_{k\to\infty}\mu_k(\mathcal K_{\delta(\epsilon)})\geq1-\epsilon
\end{align*}
where the last inequality follows from Proposition \ref{p64}. By taking $\epsilon\to0$, we conclude that $\nu$ is a probability measure.

For the second claim, we will show that $$\frac1{\Vol(\Omega_{g_k,\delta})}(g_k)_*\mu_{Ax_e}-\frac1{\Vol(\Omega_{g_k,\delta})}(g_k)_*(\mu_{Ax_e}|_{\Omega_{g_k,\delta}})\to0.$$ Let $f\in C_c(X)$. Since $f$ has compact support, there exists a small number $\delta'<\delta$ such that $$\int_X f(g_kx)d\mu_{Ax_e}(x)=\int_{\Lie(A)}f(g_k\exp(\mb{t})x_e)d\mb{t}=\int_{\Omega_{g_k,\delta'}}f(g_k\exp(\mb{t})x_e)d\mb{t}.$$ Here $d\mb{t}=dm_{\Lie(A)}(\mb{t})$ is the natural measure on $\Lie(A)$. By Corollary \ref{ac61}, we have 
\begin{eqnarray*}
&{}&\left|\frac1{\Vol(\Omega_{g_k,\delta})}\int_X f(g_kx)d\mu_{Ax_e}(x)-\frac1{\Vol(\Omega_{g_k,\delta})}\int_X f(g_kx)d\mu_{Ax_e}|_{\Omega_{g_k,\delta}}(x)\right|\\
&=&\left|\frac1{\Vol(\Omega_{g_k,\delta})}\int_{\Omega_{g_k,\delta'}} f(g_k\exp(\mb{t})x_e)d\mb{t}-\frac1{\Vol(\Omega_{g_k,\delta})}\int_{\Omega_{g_k,\delta}} f(g_k\exp(\mb{t})x_e)d\mb{t}\right|\\
&=&\left|\frac1{\Vol(\Omega_{g_k,\delta})}\int_{\Omega_{g_k,\delta'}\setminus\Omega_{g_k,\delta}} f(g_k\exp(\mb{t})x_e)d\mb{t}\right|\\
&\leq&\|f\|_\infty\frac{\Vol(\Omega_{g_k,\delta'})-\Vol(\Omega_{g_k,\delta})}{\Vol(\Omega_{g_k,\delta})}\to0.
\end{eqnarray*}
Here $\|f\|_\infty$ is the supremum of $f$. Since $(g_k)_*\mu_{Ax}$ is invariant under the action of $\Ad(g_k)A$, the probability measure $\nu$ is invariant under the action of $\lim_{k\to\infty}\Ad(g_k)A$, which is a unipotent subgroup by our assumption on $\{g_k\}_{k\in\mathbb N}$.
\end{proof}

\section{Nondivergence in terms of adjoint representations}\label{ad-nondiv}
In this section, we rewrite section \ref{nondiv} in terms of adjoint representations. The reason of doing this is that we can then apply Ratner's theorem for unipotent actions on homogeneous spaces.

Let $\Ad:G\to\SL(\mathfrak g)$ be the adjoint representation of $G=\SL(n,\mathbb R)$. The Lie algebra $\mathfrak g=\mathfrak{sl}(n,\mathbb R)$ has a $\mathbb Q$-basis $$\mathcal B=\{E_{ij}:1\leq i\neq j\leq n\}\cup\{E_{ii}:1\leq i\leq n-1\}$$ where $E_{ij}$ $(i\neq j)$ is the matrix with only nonzero entry $1$ in the $i$th row and the $j$th column, and $E_{ii}$ $(1\leq i\leq n-1)$ is the diagonal matrix with $1$ in the $(i,i)$-entry and $-1$ in the $(i+1,i+1)$-entry. We will also consider the representations $\wedge^l\Ad:G\to\SL(\wedge^l\mathfrak g)$ for $1\leq l\leq\dim\mathfrak g-1$. The set of all $l$-th wedge products of vectors in $\mathcal B$ is then a $\mathbb Q$-basis of $\wedge^l\mathfrak g$, which we denote by $\mathcal B_l$.

Let $1\leq l\leq\dim\mathfrak g-1$. For $\wedge^l\mathfrak g$, its decomposition with respect to the action of $\wedge^l\Ad A$ is given by $$\wedge^l\mathfrak g=\sum_\chi\mathfrak g_\chi$$ where each $\chi$ is a linear functional on $\Lie(A)$ such that for any $\mb{t}\in\Lie(A)$ and $v\in\mathfrak g_\chi$ $$\wedge^l\Ad(\exp(\mb{t}))v=\exp(\chi(\mb{t}))v.$$ We denote by $\mathcal W_l(\mathfrak g)$ the collection of all such linear functionals $\chi$, and let $$\mathcal W(\mathfrak g)=\bigcup_{l=1}^{\dim\mathfrak g-1}\mathcal W_l(\mathfrak g).$$ We know that each $\mathfrak g_\chi$ $(\chi\in\mathcal W_l(\mathfrak g))$ has a $\mathbb Q$-basis from $\mathcal B_l$, and we denote by $\mathfrak g_\chi(\mathbb Z)$ the subset of integer vectors with respect to this basis.  

Now let $g\in G$. We define for $gA\Gamma$ another convex polytope in $\Lie(A)$ in terms of adjoint representations, which is similar to the convex polytope $\Omega_{g,\delta}$ in section \ref{con}. Let $1\leq l\leq\dim\mathfrak g-1$ and $\chi\in\mathcal W_l(\mathfrak g)$. Let $v\in\mathfrak g_\chi(\mathbb Z)\setminus\{0\}$. Then for $\mb{t}\in\Lie(A)$, the vector $$\wedge^l\Ad(g\exp(\mb{t}))v=e^{\chi(\mb{t})}\wedge^l\Ad(g)v\notin B_\delta$$ if and only if $$\chi(\mb{t})\geq\ln\delta-\ln\|\wedge^l\Ad(g)v\|.$$  Here $B_\delta$ denotes the ball of radius $\delta>0$ around 0 with the norm $\|\cdot\|$ on $\wedge^l\mathfrak g$ induced by the norm $\|\cdot\|_{\mathfrak g}$ on $\mathfrak g$. Now we give the following

\begin{definition}
For any $g\in G$ and $\delta>0$, we denote by $R_{g,\delta}$ the subset of points $\mb{t}\in\Lie(A)$ satisfying $$\chi(\mb{t})\geq\ln\delta-\ln\|\wedge^l\Ad(g)v\|$$ for any $v\in\mathfrak g_\chi(\mathbb Z)\setminus\{0\}$, $\chi\in\mathcal W_l(\mathfrak g)$ and $1\leq l\leq\dim\mathfrak g-1.$
\end{definition}

 The proof of the following proposition is similar to that of Lemma \ref{p31}.
\begin{proposition}\label{p71}
The subset $R_{g,\delta}$ is a bounded convex polytope in $\Lie(A)$ for any $g\in G$ and $\delta>0$.
\end{proposition}

Here we list some properties about the convex polytopes $R_{g_k,\delta}$ $(k\in\mathbb N)$, which are parallel to those in section \ref{recon} and section \ref{nondiv}.

\begin{proposition}\label{p72}
Let $\delta>0$. We have
\begin{enumerate}
\item For any $\epsilon>0$ there exists $\delta(\epsilon)>0$ such that for sufficiently large $k>0$ $$m_{\Lie(A)}(R_{g_k,\delta(\epsilon)}\cap\Omega_{g_k,\delta})\geq(1-\epsilon)m_{\Lie(A)}(\Omega_{g_k,\delta}).$$
\item For sufficiently large $k$, $R_{g_k,\delta}$ contains a ball of radius $r_k>0$, and $r_k\to\infty$ as $k\to\infty$.
\end{enumerate}
\end{proposition}
\begin{proof}
For any $\epsilon>0$, let $\delta(\epsilon)$ be as in Proposition \ref{p64}. By applying Mahler's compactness criterion on the space of unimodular lattices in $\wedge^l\mathfrak g$ $(1\leq l\leq\dim\mathfrak g-1)$, we can find a $\delta'(\epsilon)>0$ such that $$\{\mb{t}\in\Omega_{g_k,\delta}: g_k\exp(\mb{t})\mathbb Z^n\in\mathcal K_{\delta(\epsilon)}\}\subset R_{g_k,\delta'(\epsilon)}\cap\Omega_{g_k,\delta}.$$ Now the first part of the proposition follows from Proposition \ref{p64}. 

For the second part, we fix $\epsilon>0$. By Lemma \ref{p33}, Lemma \ref{p52} and the first claim of the proposition, for sufficiently large $k\in\mathbb N$, the convex polytope $R_{g_k,\delta(\epsilon)}\supset R_{g_k,\delta(\epsilon)}\cap\Omega_{g_k,\delta}$ contains a ball of radius $r_k$, and $r_k\to\infty$ as $k\to\infty$. By definition, the same holds for $R_{g_k,\delta}$ for any $\delta>0$.
\end{proof}

\begin{proposition}\label{p73}
For any $\delta>0$, we have $$\lim_{k\to\infty}\frac{\Area(\partial R_{g_k,\delta})}{\operatorname{Vol}(R_{g_k,\delta})}=0.$$
\end{proposition}
\begin{proof}
The proof is identical to that of Proposition \ref{p53}.
\end{proof}

\begin{proposition}\label{p74}
Let $\delta>0$. For any $\epsilon>0$, there exists a constant $\delta(\epsilon)>0$ such that for sufficiently large $k$ $$m_{\Lie(A)}(\{\mb{t}\in R_{g_k,\delta}: g_k\exp(\mb{t})\mathbb Z^n\notin\mathcal K_{\delta(\epsilon)}\})\leq\epsilon m_{\Lie(A)}(R_{g_k,\delta}).$$
\end{proposition}
\begin{proof}
It is similar to Proposition \ref{p64}, except that we replace the linear functionals $\omega_I(\mb{t})$ by $\chi$ in $\mathcal W_l(\mathfrak g)$ $(1\leq l\leq\dim\mathfrak g-1)$.
\end{proof}

\begin{proposition}\label{p75}
Let $\delta>0$. By passing to a subsequence, the sequence $\frac1{\Vol(R_{g_k,\delta})}(g_k)_*(\mu_{Ax}|_{R_{g_k,\delta}})$ converges to a probability measure $\nu$. We also have $$\frac1{\Vol(R_{g_k,\delta})}(g_k)_*\mu_{Ax}\to\nu$$ and hence the sequence $[g_k\mu_{Ax}]$ converges to $[\nu]$. Furthermore, the probability measure $\nu$ is invariant under the action of the unipotent subgroup $\lim_{n\to\infty}\Ad(g_k)A$.
\end{proposition}
\begin{proof}
It is identical to Proposition \ref{p65} with $\Omega_{g_k,\delta}$ replaced by $R_{g_k,\delta}$. 
\end{proof}

The following is an immediate corollary of Proposition \ref{p25}, Proposition \ref{p65} and Proposition \ref{p75}.

\begin{corollary}\label{c71}
For any $\delta>0$, we have $$\lim_{k\to\infty}\frac{\Vol(\Omega_{g_k,\delta})}{\Vol(R_{g_k,\delta})}=1.$$
\end{corollary}
In the rest of the paper, we will fix a $\delta>0$ for $R_{g_k,\delta}$ $(k\in\mathbb N)$ unless otherwise specified.

\section{Ratner's theorem and linearization}\label{line}
Because of Proposition \ref{p75}, we can apply measure classification theorem for unipotent actions on homogeneous spaces. This theorem was first conjectured by Raghunathan and Dani~\cite{Dani81}, and later a breakthrough was made by Margulis in his celebrated proof of the Oppenheim conjecture \cite{M}. Afterwards, the measure classification theorem was proved by Ratner in her seminal work \cite{R90MR, R90SMR, R91RMC}. One could also read a paper of Margulis and Tomanov \cite{MT} for a different proof. In this section, for convenience, we borrow the framework and the presentation of \cite{MS95}. Readers may refer to \cite{DM93} and \cite{Sha91} for related discussions. This section is the final step of preparation for the proof of Theorem \ref{nth11}, and is devoted to proving Proposition \ref{p86}.

\subsection{Prerequisites}
We start by recalling some well-known results. One could read \cite{MS95} for more details. Let $\mathcal H$ be the countable collection of all closed connected subgroups $H$ of $G$ such that $H\cap\Gamma$ is a lattice in $H$ and the group generated by one-parameter unipotent subgroups in $H$ acts ergodically on $H\Gamma/\Gamma$ with respect to the $H$-invariant probability measure.

Let $W=\lim_{k\to\infty}\Ad(g_k)A$. By our assumptions on $\{g_k\}_{k\in\mathbb N}$, $W$ is a connected unipotent subgroup of $G$. Let $\pi:G\to G/\Gamma$ be the natural projection map. For $H\in\mathcal H$, define $$N(H,W)=\{g\in G:W\subset gHg^{-1}\},\;S(H,W)=\bigcup_{H'\in\mathcal H,H'\subsetneq H}N(H',W)$$ $$T_H(W)=\pi(N(H,W))\backslash \pi(S(H,W)).$$ For any $H_1,H_2\in\mathcal H$, $T_{H_1}(W)$ and $ T_{H_2}(W)$ intersect if and only if $ T_{H_1}(W)=T_{H_2}(W)$.

\begin{theorem}[{\cite{R91RMC}, \cite[Theorem 2.2]{MS95}}]\label{th82}
Let $\mu$ be a $W$-invariant probability measure on $X$. For any $H\in\mathcal H$, let $\mu_{H,W}$ be the restriction of $\mu$ on $T_H(W)$.
\begin{enumerate}
\item One has $\mu=\sum_{H\in\mathcal H^*}\mu_{H,W}$. Here $\mathcal H^*$ is a set of representatives of $\Gamma$-conjugacy classes in $\mathcal H$.
\item For each $H\in\mathcal H^*$, $\mu_{H,W}$ is $W$-invariant. Any $W$-invariant ergodic component of $\mu_{H,W}$ is the invariant probability measure on $gH\Gamma/\Gamma$ for some $g\in N(H,W)$.
\end{enumerate}
\end{theorem}

In the following, we will fix a subgroup $H\in\mathcal H$ $(H\not=G)$. Let $d_H=\dim\Lie(H)$ and $V_H=\wedge^{d_H}\mathfrak g$. Then $G$ acts on $V_H$ via the wedge product representation $\wedge^{d_H}\Ad$. Since $H$ is a $\mathbb Q$-group, one can find an integral point $p_H\in\wedge^{d_H}\Lie(H)\setminus\{0\}$. We will fix this $p_H$. Let $N(H)$ be the normalizer of $H$ in $G$, and $\Gamma_H=N(H)\cap\Gamma$. Then $\Gamma_H\cdot p_H\subset\{p_H,-p_H\}$. Define $\overline V_H=V_H/\{1,-1\}$ if $\Gamma_H\cdot p_H=\{p_H,-p_H\}$, and $\overline V_H=V_H$ if $\Gamma_H\cdot p_H=p_H$. The action of $G$ on $V_H$ induces an action on $\overline V_H$, and we define by $$\overline\eta_H(g)=g\cdot\overline p_H$$ the orbit map $\overline\eta_H:G\to\overline V_H$,  where $\overline p_H$ is the image of $p_H$ in $\overline V_H$. Since $\bar p_H$ is an integral point, the orbit $\Gamma\cdot\overline p_H$ is discrete in $\overline V_H$. Let $L_H$ be the Zariski closure of $\overline\eta_H(N(H,W))$ in $\overline V_H$. By \cite[Proposition 3.2]{DM93}, we have $\overline\eta_H^{-1}(L_H)=N(H,W)$.

\begin{proposition}[{\cite[Proposition 3.2]{MS95}}]\label{p82}
Let $D$ be a compact subset of $L_H$. Let $$S(D)=\{g\in\overline\eta_H^{-1}(D): g\gamma\in\overline\eta_H^{-1}(D)\text{ for some }\gamma\in\Gamma\setminus\Gamma_H\}.$$ Then $S(D)\subset S(H,W)$ and $\pi(S(D))$ is closed in $X$. Moreover, for any compact subset $\mathcal K\subset X\setminus\pi(S(D))$, there exists a neighbourhood $\Phi$ of $D$ in $\overline V_H$ such that for any $y\in\pi(\overline\eta_H^{-1}(\Phi))\cap\mathcal K$, the set $\overline\eta_H(\pi^{-1}(y))\cap\Phi$ is a singleton.
\end{proposition}

\subsection{Proof of Proposition \ref{p86}}
Now we begin to prove Proposition \ref{p86}. Let $\{f_1,f_2,\dots,f_m\}$ be a set of polynomials defining $L_H$ in $\overline V_H$. In the rest of the section, we will fix a unit vector $\vec v\in\Lie(A)$ such that all the linear functionals $\chi\in\mathcal W(\mathfrak g)$ are different on $\vec v$. One can find $d\in\mathbb N$ and $\lambda>0$ such that for any $x_0\in\Lie(A)$, the functions of $ t\in\mathbb R$ $$\|(g_k\exp(x_0+t\vec{v})\cdot w\|^2,\quad f_j(g_k\exp(x_0+t\vec{v})\cdot w),\quad1\leq j\leq m$$ belong to $E(d,\lambda)$ as defined in Definition \ref{def71}. Here the norm $\|\cdot \|$ on $\overline V_H$ is induced by the norm $\|\cdot \|_{\mathfrak g}$ on $\mathfrak g$.  We write $\delta_0$ for the constant $\delta_0(d,\lambda)$ defined in Proposition \ref{p61}.

\begin{proposition}[Cf. {\cite[Proposition 4.2]{DM93}}]\label{p83}
Let $C$ be a compact subset in $L_H$ and $\epsilon>0$. Then there exists a compact subset $D$ in $L_H$ with $C\subset D$ such that for any neighborhood $\Phi$ of $D$ in $\overline V_H$, there exists a neighborhood $\Psi$ of $C$ in $\overline V_H$ with the following property. For $x_0\in\Lie(A)$, $w\in\overline V_H$, $\Xi\subset[0,\delta_0]$ and $k\in\mathbb N$, if $\{g_k\exp(x_0+t\vec{v})\cdot w: t\in \Xi\}\not\subset\Phi$, then we have 
\begin{eqnarray*}
&{}&m_{\mathbb R}(\{t\in \Xi:g_k\exp(x_0+t\vec{v})\cdot w\in\Psi\})\\
&\leq&\epsilon m_{\mathbb R}(\{t\in \Xi:g_k\exp(x_0+t\vec{v})\cdot w\in\Phi\}).
\end{eqnarray*}
\end{proposition}
\begin{proof}
Let $d$ and $\lambda$ be defined as above. We choose a ball $B_0(r)$ of radius $r>0$ centered at 0 in $\overline V_H$ such that the closure $\overline C\subset B_0(r)$. Now for $\epsilon>0$, let $M>0$ be the constant as in Proposition \ref{p61}. Denote by $B_0(M^{\frac12}r)$ the ball of radius $M^\frac12r>0$ centered at 0. Then we take $$D:=B_0(M^{\frac12}r)\cap L_H,$$ and we will prove the proposition for this $D$.

Indeed, for any neighborhood $\Phi$ of $D$ in $\overline V_H$, one can find $\alpha>0$ such that 
$$\{u\in\overline V_H:\|u\|\leq M^\frac12r,\; |f_j(u)|\leq\alpha\;(1\leq j\leq m)\}\subset\Phi.$$ Define $$\Psi:=\{u\in\overline V_H:\|u\|< r,\;|f_j(u)|<\alpha/M\}$$ which is a neighborhood of $C$ in $\overline V_H$, and contained in $\Phi$. We show that $\Phi$ and $\Psi$ satisfy the desired property. 

Suppose $$\{g_k\exp(x_0+t\vec{v})\cdot w: t\in \Xi\}\not\subset\Phi$$ for $x_0\in\Lie(A)$, $w\in\overline V_H$ and $\Xi\subset[0,\delta_0]$.
Denote by $\mathfrak I$ the following closed subset $$\{t\in \Xi:\|g_k\exp(x_0+t\vec{v})\cdot w\|\leq M^\frac12r, |f_j(g_k\exp(x_0+t\vec{v})\cdot w)|\leq\alpha\;(1\leq j\leq m)\}.$$ One can write $\mathfrak I$ as a disjoint union of the connected components $I_i$ of $\mathfrak I$ $$\mathfrak I=\bigcup I_i.$$ On each $I_i$, we have either $$\sup_{ t\in I_i}\|g_k\exp(x_0+t\vec{v})\cdot w\|^2=Mr^2$$ or $$\sup_{ t\in I_i}|f_j(g_k\exp(x_0+t\vec{v})\cdot w)|=\alpha$$ for some $1\leq j\leq m.$ Since $\|g_k\exp(x_0+t\vec{v})\cdot w\|^2$ and $f_j(g_k\exp(x_0+t\vec{v})\cdot w)\;(1\leq j\leq m)$ belong to $E(d,\lambda)$, by Proposition \ref{p61} and the definition of $\Psi$, we obtain $$m_{\mathbb R}(\{t\in I_i:g_k\exp(x_0+t\vec{v})\cdot w\in\Psi\})\leq\epsilon m_{\mathbb R}(I_i).$$ Now we compute
\begin{align*}
&m_{\mathbb R}(\{t\in \Xi:g_k\exp(x_0+t\vec{v})\cdot w\in\Psi\})\\
=&m_{\mathbb R}(\{t\in \mathfrak I:g_k\exp(x_0+t\vec{v})\cdot w\in\Psi\})\\
=&\sum_im_{\mathbb R}(\{t\in I_i:g_k\exp(x_0+t\vec{v})\cdot w\in\Psi\})\\
\leq&\sum_i\epsilon m_{\mathbb R}(I_i)=\epsilon m_{\mathbb R}(\mathfrak I)\\
\leq&\epsilon m_{\mathbb R}(\{t\in \Xi:g_k\exp(x_0+t\vec{v})\cdot w\in\Phi\});
\end{align*}
This completes the proof of the proposition.
\end{proof}

For our purpose, we define a convex polytope in $R_{g_k,\delta}$ as follows. By Proposition \ref{p73}, we know that $$\lim_{k\to\infty}\frac{\Area(\partial R_{g_k,\delta})}{\operatorname{Vol}(R_{g_k,\delta})}=0.$$ Therefore, for each $k\in\mathbb N$, we can find a constant $d_k>0$ such that $$\lim_{k\to\infty}d_k=\infty\text{ and }\lim_{k\to\infty}\frac{d_k\Area(\partial R_{g_k,\delta})}{\operatorname{Vol}(R_{g_k,\delta})}=0.$$ Then we denote by $R'_{g_k,\delta}$ the subset of points $\mb{t}\in\Lie(A)$ satisfying $$\chi(\mb{t})\geq\ln\delta+d_k-\ln\|\wedge^l\Ad(g_k)v\|$$ for any $v\in\mathfrak g_\chi(\mathbb Z)\setminus\{0\},\chi\in\mathcal W_l(\mathfrak g),1\leq l\leq\dim\mathfrak g-1.$ This is a convex polytope inside $R_{g_k,\delta}$.

In the following, we list some properties about $R'_{g_k,\delta}$ $(k\in\mathbb N)$. 
\begin{lemma}\label{p84}
Let $d_k$ and $R'_{g_k\delta}$ be as defined above.
\begin{enumerate}
\item We have $$\lim_{k\to\infty}\frac{\operatorname{Vol}(R'_{g_k,\delta})}{\operatorname{Vol}(R_{g_k,\delta})}=1.$$
\item For $1\leq l\leq\dim\mathfrak g-1$, a functional $\chi\in\mathcal W_l(\mathfrak g)$, a nonzero $v\in\mathfrak g_\chi(\mathbb Z)$ and $k\in\mathbb N$, we have $$\|e^{\chi(\mb{t})}(\wedge^l\Ad(g_k)v)\|\geq\delta e^{d_k}\;(\forall\mb{t}\in R'_{g_k,\delta}).$$
\item For any $x_0$ in the $\delta_0$-neighborhood $R'_{g_k,\delta}$ and for the interval $\Xi=[0,\delta_0]$, there exists a constant $c>0$ which depends only on the linear functionals in $\mathcal W(\mathfrak g)$ and $\delta_0$, such that for any nonzero integer vector $w\in\wedge^l\mathfrak g$ $(1\leq l\leq\dim\mathfrak g-1)$, one has $$\sup_{t\in \Xi}\|\wedge^l(\Ad(g_k\exp(x_0+t\vec{v})))\cdot w\|\geq c\delta e^{d_k}.$$
\end{enumerate}
\end{lemma}
\begin{proof}
 The proof of the first claim is similar to that of Corollary \ref{ac61}. Indeed, let $\{f_i\}$ be the collection of the facets of $R_{g_k,\delta}$, and denote by $P_i$ the hyperplane determined by $f_i$. For each $f_i$, let $B_i$ be the unique cylinder with the following properties:
\begin{enumerate}
\item[(a)] the base of $B_i$ is $f_i$, and the height of $B_i$ is equal to $d_k$.
\item[(b)] $B_i$ and $R_{g_k,\delta}$ lie in the same half-space determined by $P_i$.
\end{enumerate}
Then one has $$\operatorname{Vol}(R_{g_k,\delta})=\bigcup_i B_i\cup\operatorname{Vol}(R'_{g_k,\delta})$$ and $$\operatorname{Vol}(R_{g_k,\delta})\leq\sum_i\operatorname{Vol}(B_i)+\operatorname{Vol}(R'_{g_k,\delta}))=d_k\Area(\partial R_{g_k,\delta})+\operatorname{Vol}(R'_{g_k,\delta}).$$ Now the first claim follows from our choice of $d_k$.

 The second claim follows from the definition of $R'_{g_k,\delta}$. To prove the last statement,  we write for any nonzero integer vector $w\in\wedge^l\mathfrak g$ $$w=\sum_{\chi\in\mathcal W_l(\mathfrak g)} w_\chi$$ where $w_\chi\in\mathfrak g_\chi(\mathbb Z)$. One can compute $$(\wedge^l\Ad(g_k\exp(\mb{t})))\cdot w=\sum_\chi e^{\chi(\mb{t})}\wedge^l\Ad(g_k)w_\chi.$$ Now the last claim follows from the second claim of the lemma and Proposition \ref{p63}. 
\end{proof}

The following proposition is an important step towards Proposition \ref{p86}.
\begin{proposition}[Cf. {\cite[Proposition 3.4]{MS95}}]\label{p85}
Let $C$ be a compact subset in $L_H$ and $0<\epsilon<1$. Then there exists a closed subset $\mathcal S$ in $\pi(S(H,W))$ with the following property: for any compact set $\mathcal K\subset X\setminus\mathcal S$, there exists a neighbourhood $\Psi$ of $C$ in $\overline V_H$ such that for sufficiently large $k$, for any $x_0$ in the $\delta_0$-neighborhood of $R'_{g_k,\delta}$ and $\Xi=[0,\delta_0]$, we have $$m_{\mathbb R}(\{t\in \Xi: g_k\exp(x_0+t\vec{v})\mathbb Z^n\in\mathcal K\cap\pi(\overline\eta_H^{-1}(\Psi))\})\leq\epsilon m_{\mathbb R}(\Xi).$$ 
\end{proposition}
\begin{proof}
Let $D\subset L_H$ be a compact set as in Proposition \ref{p83} for $C$ and $\epsilon$. Then we get a closed subset $\mathcal S=\pi(S(D))$ as in Proposition \ref{p82}. Now for a compact subset $\mathcal K$ in $X\setminus\mathcal S$, let $\Phi$ be an open neighborhood of $D$ in $\overline V_H$ as in Proposition \ref{p82}. Then we have a neighborhood $\Psi$ of $C$ in $\overline V_H$ as in Proposition \ref{p83}.

By the choice of $x_0$ and Lemma \ref{p84}, for any nonzero integer vector $w\in\wedge^{d_H}\mathfrak g$ we have $$\sup_{t\in \Xi}\|g_k\exp(x_0+t\vec{v})\cdot w\|\geq c\delta e^{d_k}$$ for some $c>0$ depending only on $\mathcal W(\mathfrak g)$ and $\delta_0$. Hence $$\{g_k\exp(x_0+t\vec{v})\cdot w:t\in \Xi\}\not\subset\Phi$$ for sufficiently large $k$. 

Now for any $s\in \Xi$ with $$g_k\exp(x_0+s\vec{v})\mathbb Z^n\in\mathcal K\cap\pi(\overline\eta_H^{-1}(\Psi)),$$ by Proposition \ref{p82}, there is a unique element $w_s$ in $\overline\eta_H(\Gamma)$ such that $$g_k\exp(x_0+s\vec{v})\cdot w_s\in\Psi.$$ Let $I_s=[a_s,b_s]$ be the largest closed interval in $\Xi$ containing $s$ such that
\begin{enumerate}
\item for any $t\in I_s$, we have $$g_k\exp(x_0+t\vec{v})\cdot w_s\in\overline\Phi$$
\item either $g_k\exp(x_0+a_s\vec{v})\cdot w_s\text{ or }g_k\exp(x_0+b_s\vec{v})\cdot w_s\in\overline\Phi\setminus\Phi$.
\end{enumerate}
We denote by $\mathcal F$ the collection of all these intervals $I_s$ as $s$ runs over $\Xi$ with $$g_k\exp(x_0+s\vec{v})\mathbb Z^n\in\mathcal K\cap\pi(\overline\eta_H^{-1}(\Psi)).$$ By Proposition \ref{p82}, we know that the intervals in $\mathcal F$ cover $\Xi$ at most twice. By Proposition \ref{p83}, we have
\begin{eqnarray*}
&{}&m_{\mathbb R}(t\in \Xi:g_k\exp(x_0+t\vec{v})\mathbb Z^n\in\mathcal K\cap\pi(\overline\eta_H^{-1}(\Psi)))\\
&\leq&\sum_{I_s\in\mathcal F} m_{\mathbb R}(t\in I_s:g_k\exp(x_0+t\vec{v})\cdot w_s\in\Psi)\\
&\leq&\sum_{I_s\in\mathcal F}\epsilon m_{\mathbb R}(t\in I_s:g_k\exp(x_0+t\vec{v})\cdot w_s\in\Phi)\\
&\leq&\epsilon\sum_{I_s\in\mathcal F} m_{\mathbb R}(I_s)\leq2\epsilon m_{\mathbb R}(\Xi).
\end{eqnarray*}
This completes the proof of the proposition.
\end{proof}

\begin{proposition}\label{p86}
Let $C$ be a compact set in $L_H$ and $0<\epsilon<1$. Then there exists a closed subset $\mathcal S$ in $\pi(S(H,W))$ with the following property: for any compact set $\mathcal K\subset X\setminus\mathcal S$, there exists a neighbourhood $\Psi$ of $C$ in $\overline V_H$ such that for sufficiently large $k>0$ we have $$m_{\Lie(A)}(\{\mb{t}\in R_{g_k,\delta}: g_k\exp(\mb{t})\mathbb Z^n\in\mathcal K\cap\pi(\overline\eta_H^{-1}(\Psi))\})\leq\epsilon m_{\Lie(A)}(R_{g_k,\delta}).$$
\end{proposition}
\begin{proof}
By Lemma \ref{p84}, let $k$ be sufficiently large such that $$\frac{m_{\Lie(A)}(R_{g_k,\delta}\setminus R'_{g_k,\delta})}{m_{\Lie(A)}(R_{g_k,\delta})}\leq\frac\epsilon2.$$ We can find a cover of the region $R'_{g_k,\delta}$ by countably many disjoint small boxes of diameter at most $\delta_0$ such that each box is of the form $$B=\{x_0+t\vec{v}: x_0\in S\text{ and } t\in \Xi\}$$ where $S$ is the base of $B$ perpendicular to $\vec v$, and $\Xi=[0,\delta_0]$. Denote by $\mathcal F$ the collection of these boxes.

For any $B\in\mathcal F$, and for any $x_0$ in the base $S$ of $B$, $x_0$ is in the $\delta_0$-neighborhood of $R'_{g_k,\delta}$. By Proposition \ref{p85} we obtain that $$m_{\mathbb R}(\{t\in \Xi: g_k\exp(x_0+t\vec{v})\mathbb Z^n\in\mathcal K\cap\pi(\overline\eta_H^{-1}(\Psi))\})\leq\frac\epsilon2 m_{\mathbb R}(\Xi)$$ for sufficiently large $k$. By integrating the inequality above over the base $S$, one has $$m_{\Lie(A)}(\{\mb{t}\in B: g_k\exp(\mb{t})\mathbb Z^n\in\mathcal K\cap\pi(\overline\eta_H^{-1}(\Psi))\})\leq\frac\epsilon2 m_{\Lie(A)}(B).$$ By the choice of $d_k$ and $\mathcal F$, for sufficiently large $k$, we have $$\bigcup_{B\in\mathcal F}B\subset R_{g_k,\delta}.$$
Now we compute
\begin{align*}
&m_{\Lie(A)}(\{\mb{t}\in R_{g_k,\delta}: g_k\exp(\mb{t})\mathbb Z^n\in\mathcal K\cap\pi(\overline\eta_H^{-1}(\Psi))\})\\
\leq&m_{\Lie(A)}(\{\mb{t}\in R_{g_k,\delta}\setminus R'_{g_k,\delta}: g_k\exp(\mb{t})\mathbb Z^n\in\mathcal K\cap\pi(\overline\eta_H^{-1}(\Psi))\})\\
&+\sum_{B\in\mathcal F}m_{\Lie(A)}(\{\mb{t}\in B: g_k\exp(\mb{t})\mathbb Z^n\in\mathcal K\cap\pi(\overline\eta_H^{-1}(\Psi))\})\\
\leq&\frac\epsilon2m_{\Lie(A)}(R_{g_k,\delta})+\sum_{B\in\mathcal F}\frac\epsilon2m_{\Lie(A)}(B)\leq\epsilon m_{\Lie(A)}(R_{g_k,\delta}).
\end{align*}
The proposition now follows.
\end{proof}

\section{Proofs of Theorem \ref{th11}, Theorem \ref{nth11} and Theorem \ref{th12}}\label{proof}
\begin{proof}[Proof of Theorem \ref{nth11}]
We will prove the theorem by induction. Let $g_k=(u_{ij}(k))_{1\leq i,j\leq n}$ $(k\in\mathbb N)$ be a sequence in the upper triangular unipotent subgroup $N$ of $\SL(n,\mathbb R)$, and for each pair $i<j$, either $u_{ij}(k)$ is zero for all $k$ or $u_{ij}(k)\neq0$ and diverges to infinity.

Suppose for a start that $\mathcal A(A,\{g_k\}_{k\in\mathbb N})=\{0\}$. By passing to a subsequence, we may further assume that $\Ad g_k(\Lie A)$ converges to a subalgebra consisting of nilpotent elements in $\mathfrak g$, in the space of Grassmanian of $\mathfrak g$.  Then by Proposition \ref{p75}, after passing to a subsequence, $[(g_k)_*\mu_{Ax_e}]$ converges to $[\nu]$ for a probability measure $\nu$. Furthermore, we have $$\frac1{\Vol(R_{g_k,\delta})}(g_k)_*(\mu_{Ax_e}|_{R_{g_k,\delta}})\to v$$ and $\nu$ is invariant under the unipotent subgroup $W=\lim_{k\to\infty}\Ad(g_k) A$.

We will apply Ratner's theorem and the technique of linearization to prove that $\nu$ is the Haar measure on $\SL(n,\mathbb R)/\SL(n,\mathbb Z)$. According to Theorem \ref{th82}, suppose by way of contradiction that for some $H\in\mathcal H^*$ $(H\neq G)$ we have $\nu(T_H(W))>0$. Then we can find a compact subset $C\subset T_H(W)$ such that $$\nu(C)=\alpha>0.$$ Now let $0<\epsilon<\alpha$, $C_1=\overline\eta_H(C)$ and $\mathcal S$ the closed subset of $X$ as in Proposition \ref{p86}. Since $C\cap\mathcal S=\emptyset$, we can pick a compact neighborhood $\mathcal K\subset X\setminus S$ of $C$. Then by Proposition \ref{p86}, there exists a neighborhood $\Psi$ of $C_1$ in $\overline V_H$ such that for sufficiently large $k>0$ $$m_{\Lie(A)}(\{\mb{t}\in R_{g_k,\delta}: g_k\exp(\mb{t})\mathbb Z^n\in\mathcal K\cap\pi(\overline\eta_H^{-1}(\Psi))\})\leq\epsilon m_{\Lie(A)}(R_{g_k,\delta})$$ and $$C\subset \mathcal K\cap\pi(\overline\eta_H^{-1}(\Psi)).$$ This implies that $$\nu(C)\leq\epsilon<\alpha$$ which contradicts the equation $\nu(C)=\alpha$. Hence $\nu$ is the Haar measure on $\SL(n,\mathbb R)/\SL(n,\mathbb Z)$.

Now suppose that $\mathcal A(A,\{g_k\}_{k\in\mathbb N})\neq\{0\}$. Then by Corollary \ref{c42}, the subgroup $$S=\{a\in A:ag_k=g_ka\textup{ for all }k\}$$ is connected and nontrivial, and $\Lie(S)=\mathcal A(A,\{g_k\}_{k\in\mathbb N})$. This implies that all elements in $A$ and $\{g_k\}_{k\in\mathbb N}$ belong to the reductive group $C_G(S)^0$. Moreover, by the definition of $S$, $S$ is also the connected component of the center of $C_G(S)^0$. So we have $$C_G(S)^0\cong S\times H$$ where $H$ is the semisimple component of $C_G(S)^0$ and $H$ is isomorphic to the product of various $\SL(n_i,\mathbb R)$ with $n_i<n$, i.e. $$H\cong\prod\SL(n_i,\mathbb R).$$ Let $A_i=A\cap\SL(n_i,\mathbb R)$ be the connected component of the full diagonal subgroup in $\SL(n_i,\mathbb R)$, and we have $$A=S\times\prod A_i.$$ Since $g_k\in N$ is unipotent $(\forall k\in\mathbb N)$, one has $g_k\in H$. Then we can write $g_k=\prod g_{i,k}\in\prod\SL(n_i,\mathbb R)$. Note that by the definition of $S$ and Corollary \ref{c42}, $\mathcal A(A_i,\{g_{i,k}\}_{k\in\mathbb N})=\{0\}$ for all $i$.

 The above discussions tell us that our problem now can be reduced to the following setting (recall that $x_e=e\SL(n,\mathbb Z)$):
\begin{enumerate}
\item the measure $\mu_{Ax_e}$ is supported in the homogeneous space $C_G(S)^0/(\Gamma\cap C_G(S)^0)$, where one has 
\begin{align*}
C_G(S)^0/(\Gamma\cap C_G(S)^0)=&S/(\Gamma\cap S)\times H/(\Gamma\cap H)\\
=&S\times\prod(\SL(n_i,\mathbb R)/\SL(n_i,\mathbb Z)).
\end{align*}
\item the measure $\mu_{Ax_e}$ can be decomposed, according to the decomposition of $C_G(S)^0/(\Gamma\cap C_G(S)^0)$, as $$\mu_{Ax_e}=\mu_{S}\times\prod\mu_{A_ix_i}.$$ Here $\mu_{S}$ denotes the $S$-invariant measure on $S$. For each $i$, $x_i=e\SL(n_i,\mathbb Z)$ is the identity coset in $\SL(n_i,\mathbb R)/\SL(n_i,\mathbb Z)$, and $\mu_{A_ix_i}$ denotes the $A_i$-invariant measure on $A_ix_i$ in $\SL(n_i,\mathbb R)/\SL(n_i,\mathbb Z)$.
\item the measure $\mu_{Ax_e}$ is pushed by the sequence $\{g_k\}_{k\in\mathbb N}$ in the space $C_G(S)^0/(\Gamma\cap C_G(S)^0)$ in the following manner: $$(g_k)_*\mu_{Ax_e}=\mu_S\times\prod (g_{i,k})_*\mu_{A_ix_i}.$$
\item for each $A_ix_i$ in $\SL(n_i,\mathbb R)/\SL(n_i,\mathbb Z)$, one has $\mathcal A(A_i,\{g_{i,k}\}_{k\in\mathbb N})=\{0\}$.
\end{enumerate}

Since $n_i<n$, we can now apply the induction hypothesis to the sequence $(g_{i,k})_*\mu_{A_ix_i}$, and obtain that $[g_{i,k}\mu_{A_ix_i}]$ converges to the equivalence class of the Haar measure $m_{\SL(n_i,\mathbb R)/\SL(n_i,\mathbb Z)}$ on $\SL(n_i,\mathbb R)/\SL(n_i,\mathbb Z)$. Now by putting all the measures $m_{\SL(n_i,\mathbb R)/\SL(n_i,\mathbb Z)}$ and $\mu_S$ back together in the space $\SL(n,\mathbb R)/\SL(n,\mathbb Z)$, we conclude that $[(g_k)_*\mu_{Ax_e}]$ converges to $[\mu_{C_G(S)^0x_e}]$. This completes the proof of Theorem \ref{nth11}.
\end{proof}

\begin{proof}[Proof of Theorem \ref{th11}]
We first prove the following

Claim: Let $\{u_k\}_{k\in\mathbb N}$ be a sequence in the upper triangular unipotent group $N$ of $G=\SL(n,\mathbb R)$. Then there is a subsequence $\{u_{i_k}\}_{k\in\mathbb N}$ of $\{u_k\}_{k\in\mathbb N}$ such that $$u_{i_k}=b_kv_k$$ for a bounded sequence $\{b_k\}_{k\in\mathbb N}$ in $N$, and a sequence $\{v_k\}_{k\in\mathbb N}$ in $N$ with $v_k=(v_{ij}(k))_{1\leq i,j\leq n}$ satisfying the following condition: for each pair $(i,j)$ $(1\leq i<j\leq n)$ $$\text{ either }v_{ij}(k)=0\textup{ for all $k$, or }v_{ij}(k)\to\infty \text{ as $k\to\infty$}.$$
\begin{proof}[Proof of the claim]
We proceed by induction on $n$. For $n=2$ and $G=\SL(2,\mathbb R)$, $\{u_k\}_{k\in\mathbb N}$ is a sequence in the $2\times 2$ upper triangular unipotent group. Write $u_k=(u_{ij}(k))_{1\leq i,j\leq 2}$. By passing to a subsequence, we may assume that $\{u_{12}(k)\}_{k\in\mathbb N}$ is bounded, or diverges to infinity. If $\{u_{12}(k)\}_{k\in\mathbb N}$ is bounded, then set $b_k=u_k$, and $v_k=e$ the identity matrix. If $\{u_{12}(k)\}_{k\in\mathbb N}$ diverges to infinity, then set $b_k=e$ and $v_k=u_k$. Either case we have $u_k=b_kv_k$, and the claim holds in this case.

Suppose that the claim holds for $\SL(n-1,\mathbb R)$ $(n\geq 3)$. Now let $G=\SL(n,\mathbb R)$ and $\{u_k\}_{k\in\mathbb N}$ a sequence in the $n\times n$ upper triangular unipotent group $N$. We will use the notation in section \ref{graph}, i.e. for any $g\in G$, we will denote by $(g)_{l\times l}$ the $l\times l$ submatrix in the upper left corner of $g$. 

Now write $u_k=(u_{ij}(k))_{1\leq i,j\leq n}$. Then $(u_k)_{(n-1)\times(n-1)}=(u_{ij}(k))_{1\leq i,j\leq n-1}$. By applying induction hypothesis on $(u_k)_{(n-1)\times(n-1)}$, after passing to a subsequence, one could find a bounded sequence $\{w_k\}_{k\in\mathbb N}$ in $N$ and a sequence $\{x_k\}_{k\in\mathbb N}$ in $N$ with $x_k=(x_{ij}(k))_{1\leq i,j\leq n}$ such that $$(u_k)_{(n-1)\times (n-1)}=(w_k)_{(n-1)\times (n-1)}(x_k)_{(n-1)\times (n-1)},\quad u_k=w_kx_k$$ and for each pair $(i,j)$ $(1\leq i<j\leq n-1)$ $$\text{ either }x_{ij}(k)=0\textup{ for all $k$, or }x_{ij}(k)\to\infty \text{ as $k\to\infty$}.$$ Now by passing to a subsequence, one could assume that for $1\leq i\leq n-1$ $$\textup{either $\{x_{in}(k)\}_{k\in\mathbb N}$ is bounded, or } x_{in}(k)\to\infty\textup{ as }k\to\infty.$$ By Gauss elimination, there exists a bounded sequence $y_k\in N$ and a sequence $v_k\in N$ such that $$x_k=y_kv_k$$ and the following condition holds for $v_k=(v_{ij}(k))_{1\leq i,j\leq n}$: for any $1\leq i<j\leq n$, $$\text{ either } v_{ij}(k)=0\textup{ for all $k$, or }v_{ij}(k)\to\infty \text{ as $k\to\infty$}.$$ Now we complete the proof of the claim by setting $v_k$ as above and $b_k=w_ky_k$.
\end{proof}

Now we prove Theorem \ref{th11}. By Iwasawa decomposition, for each element $g_k$ in the sequence $\{g_k\}_{k\in\mathbb N}$, we can write $$g_k=s_ku_ka_k$$ where $s_k\in K=\operatorname{SO}(n,\mathbb R)$, $u_k\in N$ and $a_k\in A$. By the claim above, we can assume that, after passing to a subsequence, one could write $$u_k=b_k\tilde u_k$$ for a bounded sequence $b_k\in N$ and a sequence $\tilde u_k=(\tilde u_{ij}(k))_{1\leq i,j\leq n}$ in $N$ such that for each pair $1\leq i<j\leq n$ $$\text{ either }\tilde u_{ij}(k)=0\textup{ for all $k$, or }\tilde u_{ij}(k)\to\infty \text{ as $k\to\infty$}.$$ Since $\mu_{Ax}$ is $A$-invariant, we have $$(g_k)_*\mu_{Ax}=(s_kb_k\tilde u_k)_*\mu_{Ax}.$$ Now the first paragraph of Theorem \ref{th11} follows by applying Theroem \ref{nth11} to $(\tilde u_k)_*\mu_{Ax}$ and the boundedness of $\{b_k\}_{k\in\mathbb N}$ and $\{s_k\}_{k\in\mathbb N}$.

Now we prove the second paragraph of Theorem \ref{th11}. Assume that for any $Y\in\Lie(A)\setminus\mathcal A(A,\{g_k\}_{k\in\mathbb N})$, $\{\Ad(g_k)Y\}_{k\in\mathbb N}$ diverges. Let $[\nu]$ be a limit point of $\{[(g_k)_*\mu_{Ax}]\}_{k\in\mathbb N}$. Then there is a subsequence $\{g_{i_k}\}_{k\in\mathbb N}$ such that $[(g_{i_k})_*\mu_{Ax}]$ converges to $[\nu]$. By the same argument as above, after passing to a subsequence of $\{g_{i_k}\}$, one can find $s_k\in K$, $\tilde u_k=(\tilde  u_{ij}(k))_{1\leq i,j\leq n}\in N$ $a_k\in A$ and a bounded sequence $b_k\in N$ such that $$g_{i_k}=s_kb_k \tilde u_ka_k,$$ and for any $1\leq i<j\leq n$ $$\text{ either } \tilde u_{ij}(k)=0\textup{ for all $k$, or } \tilde u_{ij}(k)\to\infty \text{ as $k\to\infty$}.$$ Since $\mu_{Ax}$ is $A$-invariant, we have $$(g_{i_k})_*\mu_{Ax}=(s_kb_k\tilde u_k)_*\mu_{Ax}.$$  Note that by the boundedness of $\{b_k\}_{k\in\mathbb N}$ and $\{s_k\}_{k\in\mathbb N}$ $$\mathcal A(A,\{g_k\}_{k\in\mathbb N})=\mathcal A(A,\{g_{i_k}\}_{k\in\mathbb N})=\mathcal A(A,\{\tilde u_k\}_{k\in\mathbb N}).$$ Now the second paragraph of Theorem \ref{th11} follows from Theorem \ref{nth11} and the boundedness of $\{b_k\}_{k\in\mathbb N}$ and $\{s_k\}_{k\in\mathbb N}$.
\end{proof}

The following is an immediate corollary from the proof of Theorem \ref{th11}, which gives an example of $\lambda_k$'s in a special case of Theorem \ref{th13}. This also generalizes the result in \cite{OS14}. We will apply this special case of Theorem \ref{th13} in the counting problem in section \ref{app}. 

\begin{corollary}[Cf. Theorem \ref{th13}]\label{c91}
Let $\{g_k\}_{k\in\mathbb N}$ be a sequence in $KN$ such that for any nonzero $Y\in\Lie(A)$, the sequence $\{\Ad(g_k)Y\}_{k\in\mathbb N}$ diverges to infinity. Then we have $$\frac1{\Vol(\Omega_{g_k,\delta})}(g_k)_*\mu_{Ax}\to m_X$$ where $m_X$ is the $G$-invariant probability measure on $X$.
\end{corollary}

In the rest of this section, we will prove Theorem \ref{th12}. Let $H$ be a connected reductive group containing $A$. It is known that up to conjugation by an element in the Weyl group of $G$, $H$ consists of diagonal blocks with each block isomorphic to $\GL(m,\mathbb R)$ with $m<n$. For convenience, we will assume that $H$ has the form of diagonal blocks, since conjugations by Weyl elements do not affect the theorem. 

 The following lemma clarifies an assumption in Theorem \ref{th12}.
\begin{lemma}\label{al101}
Let $Ax$ be a divergent orbit in $X$ and let $H$ be a connected reductive group containing $A$. Then $Hx$ is closed in $X$.
\end{lemma}
\begin{proof}
By the classification of divergent $A$-orbits of Margulis which appears in the appendix of \cite{TW03}, we may assume without loss of generality that $x$ is commensurable to $\mathbb Z^n$. Thus, it is enough to prove the lemma for $x=\mathbb Z^n$. Then the lemma follows easily for any reductive group $H$ under consideration.
\end{proof}

By reasoning in the same way as at the beginning of section \ref{con}, it is harmless to assume $x=x_e=e\SL(n,\mathbb Z)$ in the proof of Theorem \ref{th12}.

Let $P$ be the standard $\mathbb Q$-parabolic subgroup in $G$ having $H$ as (the connected component of) a Levi component. Let $U\subset N$ be the unipotent radical of $P$. We write $$H=S\times H_{ss}$$ where $S$ is the connected component of the center of $H$, and $H_{ss}$ is the semisimple component of $H$. We will denote by $A_{ss}$ the connected component of the full diagonal group in $H_{ss}$. Note that we have $$A=S\times A_{ss}.$$ By Theorem \ref{nth11}, we can find a sequence of upper triangular unipotent matrices $h_k\in H$ satisfying the dichotomy condition in Theorem \ref{nth11} such that $$\mathcal \exp(\mathcal A(A,\{h_k\}_{k\in\mathbb N}))=S,\quad C_G(\mathcal A(A,\{h_k\}_{k\in\mathbb N}))^0=H$$ $$[(h_k)_*\mu_{Ax_e}]\to[\mu_{Hx_e}]\textup{ as }k\to\infty.$$ We will fix such a sequence $\{h_k\}_{k\in\mathbb N}$. 

As $G=KUH$ where $K=\operatorname{SO}(n,\mathbb R)$, for every $g_k$ in the sequence $\{g_k\}_{k\in\mathbb N}$, we can write $$g_k=s_ku_kl_k$$ where $s_k\in K$, $u_k\in U$ and $l_k\in H$. We have $$(g_k)_*\mu_{Hx_e}=(s_ku_k)_*\mu_{Hx_e}.$$ Following the same strategy as in the proof of Theorem \ref{th11}, to prove Theorem \ref{th12}, we may assume that $g_k\in U$. Now let $g_k=(u_{ij}(k))_{1\leq i,j\leq k}\in U$. By Gauss elimination as explained in the proof of Theorem \ref{th11}, we may further assume that for each pair $i<j$, either $u_{ij}(k)$ equals $0$ for all $k$ or $u_{ij}(k)\neq0$ diverges to infinity.

\begin{proposition}\label{p91}
If $\mathcal A(S,\{g_k\}_{k\in\mathbb N})=\{0\}$, then for any subsequence $\{g_{m_k}\}_{k\in\mathbb N}$ of $\{g_k\}_{k\in\mathbb N}$ and any subsequence $\{h_{n_k}\}_{k\in\mathbb N}$ of $\{h_k\}_{k\in\mathbb N}$, we have $\Ad(g_{m_k}h_{n_k})Y\to\infty$ as $k\to\infty$ for any nonzero $Y\in\Lie(A)$.
\end{proposition}
\begin{proof}
Let $Y=Y_1+Y_2\neq0$, where $Y_1\in\Lie(S)$ and $Y_2\in\Lie(A_{ss})$. If $Y_2=0$, then $$\Ad(g_{m_k}h_{n_k})Y=\Ad(g_{m_k})Y_1$$ diverges to $\infty$ by the condition $\mathcal A(S,\{g_k\}_{k\in\mathbb N})=\{0\}$ and Corollary \ref{c42}. If $Y_2\neq0$, then we have
\begin{eqnarray*}
\Ad(g_{m_k}h_{n_k})Y&=&\Ad(g_{m_k})(Y_1+\Ad(h_{n_k}) Y_2)\\
&=&(\Ad(g_{m_k})(Y_1+\Ad(h_{n_k}) Y_2)-(Y_1+\Ad(h_{n_k}) Y_2))\\
&{}&\quad+(Y_1+\Ad(h_{n_k}) Y_2).
\end{eqnarray*}
Since $H$ normalizes $U$, we know that $$\Ad(g_{m_k})(Y_1+\Ad(h_{n_k}) Y_2)-(Y_1+\Ad(h_{n_k}) Y_2)\in\Lie(U).$$ Also $\Ad(h_{n_k})Y_2\in\Lie(H)$ and $\Ad(h_{n_k})Y_2\to\infty$ by our choice of $\{h_k\}_{k\in\mathbb N}$ and Corollary \ref{c42}. Hence $\Ad(g_{m_k}h_{n_k})Y$ diverges to $\infty$.
\end{proof}

We will fix a nonnegative function $f_0\in C_c(X)$ such that $\operatorname{supp}(f_0)$ contains the compact orbit $N\mathbb Z^n$ in $X$. This implies that for any $g\in N$ we have $$\int f_0dg_*\mu_{Ax_e}>0.$$

\begin{proposition}\label{p92}
Suppose that the subalgebra $\mathcal A(S,\{g_k\}_{k\in\mathbb N})=\{0\}$. Let $f\in C_c(X)$. Then for any $\epsilon>0$, there exists $M>0$ such that for any $m,n>M$ $$\left|\frac{\int f d(g_mh_n)_*\mu_{Ax_e}}{\int f_0 d(g_mh_n)_*\mu_{Ax_e}}-\frac{\int fdm_X}{\int f_0dm_X}\right|\leq\epsilon.$$
\end{proposition}
\begin{proof}
Suppose that there exists $\epsilon>0$ such that for any $l>0$, there are $m_l,n_l>l$ satisfying $$\left|\frac{\int f d(g_{m_l}h_{n_l})_*\mu_{Ax_e}}{\int f_0d(g_{m_l}h_{n_l})_*\mu_{Ax_e}}-\frac{\int fdm_X}{\int f_0dm_X}\right|\geq\epsilon.$$ By Proposition \ref{p91}, we know that $\Ad(g_{m_l}h_{n_l})Y\to\infty$ as $l\to\infty$ for any nonzero $Y\in\Lie(A)$. Hence by Theorem \ref{th11}, we have $$[(g_{m_l}h_{n_l})_*\mu_{Ax_e}]\to[m_X]$$ which contradicts the inequality above. This completes the proof of the proposition.
\end{proof}

\begin{proof}[Proof of Theorem \ref{th12}]
 We will prove the theorem by induction. Let $\{g_k\}_{k\in\mathbb N}$ be a sequence in $G$ and by the discussions above, we may assume that every $g_k=(u_{ij}(k))_{1\leq i,j\leq n}$ is in $U\subset N$, and for $1\leq i<j\leq n$, $u_{ij}(k)$ either equals 0 for all $k$ or diverges to infinity as $k\to\infty$.

Suppose that $\mathcal A(S,\{g_k\}_{k\in\mathbb N})=\{0\}$. Let $f\in C_c(X)$. By Proposition \ref{p92}, for any $\epsilon>0$ there exists $M>0$ such that for any $m,n>M$
$$\left|\frac{\int f d(g_mh_n)_*\mu_{Ax_e}}{\int f_0 d(g_mh_n)_*\mu_{Ax_e}}-\frac{\int fdm_X}{\int f_0dm_X}\right|\leq\epsilon.$$
Now we fix $m$, let $n\to\infty$ and obtain $$\left|\frac{\int f d(g_m)_*\mu_{Hx_e}}{\int f_0 d(g_m)_*\mu_{Hx_e}}-\frac{\int fdm_X}{\int f_0dm_X}\right|\leq\epsilon.$$ This implies that $[(g_k)_*\mu_{Hx_e}]\to[m_X]$.

Now suppose that $\mathcal A(S,\{g_k\}_{k\in\mathbb N})\neq\{0\}$. The proof in this case would be similar to that of Theorem \ref{nth11}. By Corollary \ref{c42}, the subgroup $$S'=\{a\in S:ag_k=g_ka\}$$ is connected and nontrivial, and $\Lie(S')=\mathcal A(S,\{g_k\}_{k\in\mathbb N})$. This implies that all elements of $H$ and $\{g_k\}$ belong to $C_G(S')^0$. Moreover, we have $$C_G(S')^0\cong S'\times H'$$ where $H'$ is the semisimple component of $C_G(S')^0$ and $H'$ is isomorphic to the product of various $\SL(n_i,\mathbb R)$ with $n_i<n$, $$H'\cong\prod_i\SL(n_i,\mathbb R).$$ Let $H_i$ be the reductive subgroup $H\cap\SL(n_i,\mathbb R)$ in $\SL(n_i,\mathbb R)$, and we have $$H=S'\times\prod_i H_i.$$ Since $g_k\in U$ is unipotent $(\forall k\in\mathbb N)$, one has $g_k\in H'$. Then we can write $g_k=\prod_i g_{i,k}$ ($g_{i,k}\in\SL(n_i,\mathbb R)$).

Similar to the proof of Theorem \ref{nth11}, the above discussions imply that the problem is reduced to the following setting: 
\begin{enumerate}
\item the measure $\mu_{Hx_e}$ is supported in the homogeneous space $C_G(S')^0/(\Gamma\cap C_G(S')^0)$, where one has \begin{align*}
C_G(S')^0/(\Gamma\cap C_G(S')^0)=& S'/(\Gamma\cap S')\times H'/(\Gamma\cap H')\\
=&S'\times\prod(\SL(n_i,\mathbb R)/\SL(n_i,\mathbb Z)).
\end{align*}
\item the measure $\mu_{Hx_e}$ can be decomposed, according to the decomposition of $C_G(S')^0/(\Gamma\cap C_G(S')^0)$, as $$\mu_{Hx_e}=\mu_{S'}\times\prod\mu_{H_ix_i}.$$ Here $\mu_{S'}$ denotes the $S'$-invariant measure on $S'$. For each $i$, $x_i=e\SL(n_i,\mathbb Z)$ is the identity coset in $\SL(n_i,\mathbb R)/\SL(n_i,\mathbb Z)$, and $\mu_{H_ix_i}$ denotes the $H_i$-invariant measure on $H_ix_i$ in $\SL(n_i,\mathbb R)/\SL(n_i,\mathbb Z)$.
\item the measure $\mu_{Hx_e}$ is pushed by the sequence $\{g_k\}$ in the space $C_G(S')^0/(\Gamma\cap C_G(S')^0)$ in the following way: $$(g_k)_*\mu_{Hx_e}=\mu_{S'}\times\prod (g_{i,k})_*\mu_{H_ix_i}.$$
\item if $S_i$ is the connected component of the center of $H_i$, then one has $\mathcal A(S_i,\{g_{i,k}\}_{k\in\mathbb N})=\{0\}$.
\end{enumerate}

Since $n_i<n$, we can now apply the induction hypothesis to the sequence $(g_{i,k})_*\mu_{H_ix_i}$, and obtain that $[(g_{i,k})_*\mu_{H_ix_i}]$ converges to the equivalence class of the Haar measure $m_{\SL(n_i,\mathbb R)/\SL(n_i,\mathbb Z)}$ on $\SL(n_i,\mathbb R)/\SL(n_i,\mathbb Z)$. Now by putting all the measures $[m_{\SL(n_i,\mathbb R)/\SL(n_i,\mathbb Z)}]$ and $\mu_{S'}$ back together in the space $\SL(n,\mathbb R)/\SL(n,\mathbb Z)$, we have $[(g_k)_*\mu_{Hx_e}]\to[\mu_{C_G(\mathcal A( S,\{g_k\}))^0x_e}]$.
\end{proof}

\section{An application to a counting problem}\label{app}
In this section, we will prove Theorem \ref{th14}. Let $p_0(\lambda)$ be a monic polynomial in $\mathbb Z[x]$ such that $p_0(\lambda)$ splits completely in $\mathbb Q$. Then by Gauss lemma, we have $p_0(\lambda)=(\lambda-\alpha_1)(\lambda-\alpha_2)\cdots(\lambda-\alpha_n)$ for $\alpha_i\in\mathbb Z$. We assume that $\alpha_i$ are distinct and nonzero. Let $M(n,\mathbb R)$ be the space of $n\times n$ matrices with the norm $$\|M\|^2=\operatorname{Tr}(M^tM)=\sum_{1\leq i,j\leq n} x_{ij}^2$$ for $M=(x_{ij})_{1\leq i,j\leq n}$. Note that this norm is $\Ad(K)$-invariant. We will denote by $B_T$ the ball of radius $T$ centered at 0 in $M(n,\mathbb R)$. We denote by $$M_\alpha=\diag(\alpha_1,\alpha_2,\dots,\alpha_n)\in M(n,\mathbb Z).$$ For $M\in M(n,\mathbb R)$, we denote by $p_M(\lambda)$ the characteristic polynomial of $M$. We consider $$V(\mathbb R)=\{M\in M(n,\mathbb R): p_M(\lambda)=p_0(\lambda)\}$$ and its subset of integral points $$V(\mathbb Z)=\{M\in M(n,\mathbb Z): p_M(\lambda)=p_0(\lambda)\}.$$ We would like to get an asymptotic formula for
\begin{eqnarray*}
\#|V(\mathbb Z)\cap B_T|=\#|\{M\in M(n,\mathbb Z): p_M(\lambda)=p_0(\lambda),\;\|M\|\leq T\}|.
\end{eqnarray*}

We begin with the following proposition which is a corollary of \cite{BHC62} and \cite{LM33}.

\begin{proposition}\label{p101}
We have $$\Ad(\SL(n,\mathbb R))M_\alpha=V(\mathbb R)$$ and there are finitely many $\SL(n,\mathbb Z)$-orbits in $V(\mathbb Z)$. The number of the $\SL(n,\mathbb Z)$-orbits in $V(\mathbb Z)$ is equal to the number of classes of nonsingular ideals in the ring $\mathbb Z[M_\alpha]$.
\end{proposition}

By Proposition \ref{p101}, it suffices to compute the integral points of an $\SL(n,\mathbb Z)$-orbit. In what follows, we will consider the $\SL(n,\mathbb Z)$-orbit of $M_\alpha$. We will apply Theorem \ref{th13} (more precisely, Corollary \ref{c91}) with initial point $x=e\Gamma$ to compute $$\#|\Ad(\SL(n,\mathbb Z))M_\alpha\cap B_T|.$$For any other $\SL(n,\mathbb Z)$-orbit of $M'\in V(\mathbb Z)$, there exists $M_q\in\SL(n,\mathbb Q)$ such that $$\Ad(M_q)M'=M_\alpha$$ and the treatment for $\Ad(\SL(n,\mathbb Z))M'$ would be similar, just with a change of initial point from $e\Gamma$ to $x_q=M_q\Gamma$. See also the beginning of section \ref{con}.

As explained in section \ref{defres}, the metric $\|\cdot\|_{\mathfrak g}$ on $\mathfrak g$ defines a Haar measure $\mu_A$ on $A$ and a Haar measure $\mu_N$ on $N$. Let $\mu_K$ be the $K$-invariant probability measure on $K$. Then we define a Haar measure $\mu_G$ on $G$ by Iwasawa decomposition $G=KNA$. Let $c_X$ be the volume of $X=G/\Gamma$ with respect to $\mu_G$.
 
Now let $h=(u_{ij})_{1\leq i,j\leq n}\in N$ and write $$\Ad(h)M_\alpha=hM_\alpha h^{-1}=(x_{ij})_{1\leq i,j\leq n}$$ where $x_{ii}=\alpha_i$ and $u_{ij}=0$ $(i>j)$. We have $$hM_\alpha=(x_{ij})_{1\leq i,j\leq n}h$$ and $$\alpha_ju_{ij}=\sum_kx_{ik}u_{kj},\quad(\alpha_j-\alpha_i)u_{ij}=\sum_{k\neq i}x_{ik}u_{kj}.$$ Let $q_i(x)=\prod_{k=1}^i(x-\alpha_k).$ Lemma \ref{p102} and Lemma \ref{p103} below describe the relation between $u_{ij}$ and $x_{ij}$.

\begin{lemma}\label{p102}
For $j>i$, we have $$u_{ij}=\frac1{\alpha_j-\alpha_i}x_{ij}+f_{ij}(x)$$ where $f_{ij}$ is a polynomial in variables $x_{pq}$ with $0<q-p<j-i$, and $f_{ij}=0$ for $j-i=1$. In particular, we have the change of coordinates of the Haar measure $\mu_N$ on $N$ $$\prod_{j>i} du_{ij}=\frac1{\prod_{j>i}|\alpha_j-\alpha_i|}\prod_{j>i} dx_{ij}.$$
\end{lemma}
\begin{proof}
 It is easy to see that $u_{ij}=x_{ij}=0$ $(i>j)$ and $u_{ii}=1$. We prove the proposition by induction on $j-i$. For $j-i=1$, we have $$u_{ij}=u_{j-1,j}=\frac1{\alpha_j-\alpha_{j-1}}\sum_{k\neq j-1}x_{j-1,k}u_{kj}=\frac1{\alpha_j-\alpha_{j-1}}x_{j-1,j}.$$ Now we have $$(\alpha_j-\alpha_i)u_{ij}=\sum_{k\neq i}x_{ik}u_{kj}=\sum_{i<k<j}x_{ik}u_{kj}+x_{ij}$$ where $j-k<j-i$. We complete the proof by applying the induction hypothesis on $u_{kj}$.
\end{proof}

\begin{lemma}\label{p103}
For $j>i$, we have $$u_{ij}=\prod_{k=i}^{j-1}\frac{x_{k,k+1}}{\alpha_j-\alpha_k}+f_{ij}(x)=\frac{q_{i-1}(\alpha_j)}{q_{j-1}(\alpha_j)}\prod_{k=i}^{j-1}x_{k,k+1}+f_{ij}(x)$$ where $f_{ij}(x)$ is a polynomial in variables $x_{pq}\;(p<q)$ of degree less than $j-i$.
\end{lemma}
\begin{proof}
We prove the proposition by induction on $j-i$. For $j-i=1$, we have 
$$(\alpha_j-\alpha_i)u_{ij}=(\alpha_j-\alpha_{j-1})u_{j-1,j}=\sum_{k\neq j-1}x_{j-1,k}u_{kj}=x_{j-1,j}.$$ Now we have $$(\alpha_j-\alpha_i)u_{ij}=\sum_{k\neq i}x_{ik}u_{kj}=\sum_{i<k\leq j}x_{ik}u_{kj}$$ where $j-k<j-i$. By applying the induction hypothesis on $u_{kj}$ we have
\begin{eqnarray*}
(\alpha_j-\alpha_i)u_{ij}&=&\sum_{i<k\leq j}x_{ik}\prod_{p=k}^{j-1}\frac{x_{p,p+1}}{\alpha_j-\alpha_p}+...\\
&=&x_{i,i+1}\prod_{p=i+1}^{j-1}\frac{x_{p,p+1}}{\alpha_j-\alpha_p}+...
\end{eqnarray*}
Here we omit the terms of degree less than $j-i$. This completes the proof of the proposition.
\end{proof}

\begin{lemma}\label{p104}
For any $1\leq l\leq n$ and $1\leq i_1<i_2<\cdots<i_l\leq n$ we have $$c(i_1,i_2,\dots,i_l):=\det\left(\frac{q_{k-1}(\alpha_{i_j})}{q_{i_j-1}(\alpha_{i_j})}\right)_{1\leq k\leq l, 1\leq j\leq l}\neq0.$$
\end{lemma}
\begin{proof}
By algebraic manipulations, we can rewrite the determinant above as $$\prod_{j=1}^l\frac1{q_{i_j-1}(\alpha_{i_j})}\left(\begin{array}{cccc}1 & 1 & \cdots & 1 \\q_1(\alpha_{i_1}) & q_1(\alpha_{i_2}) & \cdots & q_1(\alpha_{i_l}) \\ \vdots & \vdots & \cdots & \vdots \\q_{l-1}(\alpha_{i_1}) & q_{l-1}(\alpha_{i_2}) & \cdots & q_{l-1}(\alpha_{i_l})\end{array}\right).
$$ Since $\deg q_i=i$, by row reductions we have
\begin{eqnarray*}
\det\left(\begin{array}{cccc}1 & 1 & \cdots & 1 \\q_1(\alpha_{i_1}) & q_1(\alpha_{i_2}) & \cdots & q_1(\alpha_{i_l}) \\ \vdots & \vdots & \cdots & \vdots \\q_{l-1}(\alpha_{i_1}) & q_{l-1}(\alpha_{i_2}) & \cdots & q_{l-1}(\alpha_{i_l})\end{array}\right)=\det\left(\begin{array}{cccc}1 & 1 & \cdots & 1 \\\alpha_{i_1} & \alpha_{i_2} & \cdots & \alpha_{i_l} \\ \vdots & \vdots & \cdots & \vdots \\\alpha_{i_1}^{l-1} & \alpha_{i_2}^{l-1} & \cdots & \alpha_{i_l}^{l-1}\end{array}\right)\neq0.
\end{eqnarray*}
\end{proof}

\begin{proposition}\label{p105}
For any $h\in N$ (recall $\Ad(h)M_\alpha=(x_{ij})_{1\leq i,j\leq n}$), we have 
\begin{eqnarray*}
&{}&h(e_{i_1}\wedge e_{i_2}\wedge\cdots\wedge e_{i_l})\\
&=&c(i_1,i_2,\dots,i_l)\prod_{j=1}^l\prod_{p=j}^{i_j-1}x_{p,p+1}(e_1\wedge e_2\wedge\cdots\wedge e_l)+...
\end{eqnarray*}
Here $c(i_1,i_2,\dots,i_l)$ is the number in Lemma \ref{p104} and we omit the terms of polynomials in variables $x_{pq}\;(p<q)$ of degrees smaller than $\sum_{j=1}^l(i_j-j)$.
\end{proposition}
\begin{proof}
By Lemma \ref{p103}, we know that $u_{ij}$ is a polynomial of degree $j-i$. This implies that the term in $h(e_{i_1}\wedge e_{i_2}\wedge\cdots\wedge e_{i_l})$ corresponding to the $e_{j_1}\wedge e_{j_2}\wedge\cdots\wedge e_{j_l}$-coordinate has degree at most $i_1+i_2+\cdots+i_l-j_1-j_2-\cdots-j_l$. To prove the proposition, it suffices to prove that the term corresponding to $e_1\wedge e_2\wedge\cdots\wedge e_l$ is a polynomial with its leading term $$c(i_1,i_2,\dots,i_l)\prod_{j=1}^l\prod_{p=j}^{i_j-1}x_{p,p+1}$$ of degree $i_1+i_2+\cdots+i_l-1-2-\cdots-l$.

We know that the coefficient of $e_1\wedge e_2\wedge\cdots\wedge e_l$ is equal to $$\det(u_{k,i_j})_{1\leq k\leq l,1\leq j\leq l}$$ and by Lemma \ref{p103} we know that the leading term of this coefficient is equal to $$\det\left(\frac{q_{k-1}(\alpha_{i_j})}{q_{i_j-1}(\alpha_{i_j})}\prod_{p=k}^{i_j-1}x_{p,p+1}\right)_{1\leq k\leq l,1\leq j\leq l}.$$ The expansion formula of determinant then gives $$\sum_{\sigma\in S_l}(-1)^{\sign(\sigma)}\prod_{j=1}^l\frac{q_{\sigma(j)-1}(\alpha_{i_j})}{q_{i_j-1}(\alpha_{i_j})}\prod_{p=\sigma(j)}^{i_j-1}x_{p,p+1}$$
where $\sigma$ runs over all the permutations in the symmetric group $S_l$. Note that we have
\begin{eqnarray*}
\prod_{j=1}^l\prod_{p=\sigma(j)}^{i_j-1}x_{p,p+1}=\prod_{j=1}^l\frac{\prod_{p=1}^{i_j-1}x_{p,p+1}}{\prod_{p=1}^{\sigma(j)-1}x_{p,p+1}}=\frac{\prod_{j=1}^l\prod_{p=1}^{i_j-1}x_{p,p+1}}{\prod_{j=1}^l\prod_{p=1}^{j-1}x_{p,p+1}}=\prod_{j=1}^l\prod_{p=j}^{i_j-1}x_{p,p+1}.
\end{eqnarray*}
This implies that 
\begin{eqnarray*}
&{}&\det\left(\frac{q_{k-1}(\alpha_{i_j})}{q_{i_j-1}(\alpha_{i_j})}\prod_{p=k}^{i_j-1}x_{p,p+1}\right)_{1\leq k\leq l,1\leq j\leq l}\\
&=&\left(\sum_{\sigma\in S_l}(-1)^{\sign(\sigma)}\prod_{j=1}^l\frac{q_{\sigma(j)-1}(\alpha_{i_j})}{q_{i_j-1}(\alpha_{i_j})}\right)\prod_{j=1}^l\prod_{p=j}^{i_j-1}x_{p,p+1}\\
&=&c(i_1,i_2,\dots,i_l)\prod_{j=1}^l\prod_{p=j}^{i_j-1}x_{p,p+1}
\end{eqnarray*}
where $c(i_1,i_2,\dots,i_l)$ is the number as in Lemma \ref{p104}.
\end{proof}

Now we define $$N(T)=\{h\in N: \Ad(h)M_\alpha=(x_{ij})_{1\leq i,j\leq n}\in B_T\}$$ $$N(\epsilon,T)=\{h\in N: \Ad(h) M_\alpha=(x_{ij})_{1\leq i,j\leq n}\in B_T, |x_{i,i+1}|\geq\epsilon T\textup{ for all }i<n\}.$$

\begin{lemma}\label{p106}
We have $$\mu_N(N(T))=\frac{\Vol(B_1)}{\prod_{j>i}|\alpha_j-\alpha_i|}T^{n(n-1)/2}$$ $$\mu_N(N(T)\setminus N(\epsilon,T))=O(\epsilon T^{n(n-1)/2}).$$ Here $\Vol(B_1)$ is the volume of the unit ball in $\mathbb R^{n(n-1)/2}$.
\end{lemma}
\begin{proof}
This follows immediately from Lemma \ref{p102}.
\end{proof}

Let $\mu_{G/A}$ be the $G$-invariant measure on $G/A$. In the following, we compute the volume of $V(\mathbb R)\cap B_T$ with respect to a volume form $\mu_{V(\mathbb R)}$ on $V(\mathbb R)$ induced by a $G$-invariant measure on $G/C_G(A)$. We may assume that the natural projection map $G/A\to G/C_G(A)$ sends $\mu_{G/A}$ to $\mu_{V(\mathbb R)}$. By Iwasawa decomposition, one has $G/A\cong KN$, and it is well-known that for any $f\in C_c(G/A)$ $$\int_{G/A}fd\mu_{G/A}=\int_K\int_Nf(kh)d\mu_K(k)d\mu_N(h)$$ via this isomorphism.

\begin{proposition}\label{ap101}
The volume of $V(\mathbb R)\cap B_T$ with respect to the volume form $\mu_{V(\mathbb R)}$ equals $$\frac{\Vol(B_1)}{\prod_{j>i}|\alpha_j-\alpha_i|}T^{n(n-1)/2}.$$ Here $\Vol(B_1)$ is as in Lemma \ref{p106}.
\end{proposition}
\begin{proof}
Note that by the discussion above, one has 
\begin{align*}
\mu_{V(\mathbb R)}(V(\mathbb R)\cap B_T)=&\mu_{G/A}(\{gA: \Ad(g)M_\alpha\in B_T\})\\
=&\mu_K\times\mu_N(\{kh:\Ad(kh)M_\alpha\in B_T\}).
\end{align*}
 By Lemma \ref{p106} and the $\Ad(K)$-invariance of the norm on $M(n,\mathbb R)$, we compute
\begin{align*}
&\mu_K\times\mu_N(\{kh:\Ad(kh)M_\alpha\in B_T\})\\
=&\mu_N(\{h:\Ad(h)M_\alpha\in B_T\})=\frac{\Vol(B_1)}{\prod_{j>i}|\alpha_j-\alpha_i|}T^{n(n-1)/2}.
\end{align*}
This completes the proof of the proposition.
\end{proof}

\begin{proposition}\label{p107}
For any $k\in K$ and $h\in N(T)$ we have $$\Vol(\Omega_{kh,\delta})=O((\ln T)^{n-1})$$ where the implicit constant depends only on $\delta$ and $M_\alpha$. Furthermore, for $h\in N(\epsilon,T)$ we have $$\Vol(\Omega_{kh,\delta})=(c_0+o(1))(\ln T)^{n-1}$$ where the implicit constant depends on $\epsilon,\delta,M_\alpha$, and $c_0$ equals the volume of $$\left\{\mb{t}\in\Lie(A):\sum_{j=1}^l t_{i_j}\geq\sum_{j=1}^l(j-i_j),\forall 1\leq l\leq n,\forall1\leq i_1<\cdots<i_l\leq n\right\}.$$
\end{proposition}
\begin{proof}
From the definition of $\Omega_{kh,\delta}$, we know that $$\Omega_{kh,\delta}=\{\mb{t}\in\Lie(A): \sum_{j=1}^l t_{i_j}\geq\ln\delta-\ln\|khe_{I}\|\text{ for any nonempty }I\in\mathcal I_n\}.$$ Since $k\in\operatorname{SO}(n,\mathbb R)$, by Proposition \ref{p105}, for any $i_1<i_2<\cdots<i_l$ we have 
\begin{eqnarray*}
&{}&\ln\delta-\ln\|kh(e_{i_1}\wedge e_{i_2}\wedge\dots\wedge e_{i_l})\|\\
&\geq&O(1)-(i_1+i_2+\cdots+i_l-1-2-\cdots-l)(\ln T)
\end{eqnarray*}
where the implicit constant depends only on $\delta$ and $M_\alpha$. Moreover if $h\in N(\epsilon,T)$ then we have
\begin{eqnarray*}
&{}&\ln\delta-\ln\|kh(e_{i_1}\wedge e_{i_2}\wedge\dots\wedge e_{i_l})\|\\
&=&O(1)-(i_1+i_2+\cdots+i_l-1-2-\cdots-l)(\ln T)
\end{eqnarray*}
where the implicit constant depends only on $\epsilon$, $\delta$ and $M_\alpha$. The proposition now follows from these equations.
\end{proof}

Define $$F_T(g)=\sum_{\gamma\in\Gamma/\Gamma_{M_\alpha}}\chi_T(\Ad(g\gamma) M_\alpha)$$ where $\chi_T$ is the characteristic function of $B_T$ in $M(n,\mathbb R)$ and $\Gamma_{M_\alpha}$ is the stabilizer of $M_\alpha$ in $\Gamma$. This defines a function on $G/\Gamma$. Note that $\chi_T$ is $\Ad(K)$-invariant and $\Gamma_{M_\alpha}$ is finite. In the following proposition, we will denote by $$(f,\phi):=\int_{G/\Gamma}f(g)\phi(g)dm_X(g)$$ for any two functions $f,\phi$ on $G/\Gamma$, whenever this integral is valid. 

\begin{proposition}\label{p108}
For any $\psi\in C_c(G/\Gamma)$, we have $$\left(\frac{|\Gamma_{M_\alpha}|c_X}{n_0T^{n(n-1)/2}(\ln T)^{n-1}}F_T,\psi\right)\to(1,\psi).$$ Here $$n_0=\frac{c_0\Vol(B_1)}{\prod_{j>i}|\alpha_j-\alpha_i|}$$ and $c_0$ is the number as in Proposition \ref{p107} and $\Vol(B_1)$ is as in Lemma \ref{p106}.
\end{proposition}
\begin{proof}
We have
\begin{eqnarray*}
(F_T,\psi)&=&\frac1{|\Gamma_{M_\alpha}|}\int_{G/\Gamma} \sum_{\gamma\in\Gamma}\chi_T(\Ad(g\gamma) M_\alpha)\psi(g) d m_X\\
&=&\frac1{|\Gamma_{M_\alpha}|c_X}\int_G\chi_T(\Ad(g)M_\alpha)\psi(g) d\mu_G(g)\\
&=&\frac1{|\Gamma_{M_\alpha}|c_X}\int_{KN}\int_A\chi_T(\Ad(kha)M_\alpha)\psi(kha) d\mu_Kd\mu_Nd\mu_A\\
&=&\frac1{|\Gamma_{M_\alpha}|c_X}\int_K\int_N\chi_T(\Ad(h)M_\alpha)d\mu_Nd\mu_K\int_A\psi(kha) d\mu_A.
\end{eqnarray*}
Now fix $\epsilon>0$. By Corollary \ref{c91}, we proceed
\begin{eqnarray*}
&=&\frac1{|\Gamma_{M_\alpha}|c_X}\int_K\int_{N(\epsilon,T)}\chi_T(\Ad(h)M_\alpha)\Vol(\Omega_{kh,\delta})d\mu_Nd\mu_K\frac1{\Vol(\Omega_{kh,\delta})}\int_A\psi(kha) d\mu_A\\
&{}&+\frac1{|\Gamma_{M_\alpha}|c_X}\int_K\int_{N\setminus N(\epsilon,T)}\chi_T(\Ad(h)M_\alpha)d\mu_Nd\mu_K\int_A\psi(kha) d\mu_A\\
&=&\frac1{|\Gamma_{M_\alpha}|c_X}\int_K\int_{N(\epsilon,T)}\chi_T(\Ad(h)M_\alpha)\Vol(\Omega_{kh,\delta})d\mu_Nd\mu_K\left(\int_{G/\Gamma}\psi d m_X+o_\epsilon(1)\right)\\
&{}&+\frac1{|\Gamma_{M_\alpha}|c_X}\int_K\int_{N\setminus N(\epsilon,T)}\chi_T(\Ad(h)M_\alpha)d\mu_Nd\mu_K\int_A\psi(kha) d\mu_A.
\end{eqnarray*}
Note that since $\psi\in C_c(G/\Gamma)$ we can find $\delta_\psi>0$ such that $$\int_A\psi(kha) d\mu_A=\int_{\Omega_{kh,\delta_\psi}}\psi(kha) d\mu_A.$$ So by Lemma \ref{p106} and Proposition \ref{p107}, we proceed
\begin{eqnarray*}
&=&\frac1{|\Gamma_{M_\alpha}|c_X}\int_K\int_{N(\epsilon,T)}\chi_T(\Ad(h)M_\alpha)\Vol(\Omega_{kh,\delta})d\mu_Nd\mu_K\int_{G/\Gamma}\psi d m_X\\
&{}&+o_\epsilon(T^{n(n-1)/2}(\ln T)^{n-1})+O_\psi(\epsilon T^{n(n-1)/2}(\ln T)^{n-1})\\
&=&\frac{n_0T^{n(n-1)/2}(\ln T)^{n-1}}{|\Gamma_{M_\alpha}|c_X}\int_{G/\Gamma}\psi d m_X\\
&{}&+o_{\epsilon,\delta}(T^{n(n-1)/2}(\ln T)^{n-1})+O_\psi(\epsilon T^{n(n-1)/2}(\ln T)^{n-1}).
\end{eqnarray*}
This implies that $$\limsup_{T\to\infty}\left|\left(\frac{|\Gamma_{M_\alpha}|c_X}{n_0T^{n(n-1)/2}(\ln T)^{n-1}}F_T,\psi\right)-(1,\psi)\right|\leq O_\psi(\epsilon).$$ We complete the proof by letting $\epsilon\to0$.
\end{proof} 

\begin{proof}[Proof of Theorem \ref{th14}]
Following the same proofs as in \cite{DRS93} and \cite{EMS96}, and combining Lemma \ref{p106} and Proposition \ref{p108}, we conclude that $$\frac{|\Gamma_{M_\alpha}|c_X}{n_0T^{n(n-1)/2}(\ln T)^{n-1}}F_T\to1.$$ Now Theorem \ref{th14} follows from this equation and Proposition \ref{p101} 
\end{proof}

\bibliographystyle{amsalpha}

\end{document}